\documentclass[draft]{amsart}

\usepackage{amsmath,amssymb}
\usepackage{amsthm, amsrefs}
\usepackage{latexsym}
\usepackage{indentfirst}
\usepackage{mathrsfs}
\usepackage{placeins}
\usepackage{graphicx}
\usepackage{mathtools}

\numberwithin{equation}{section}

\theoremstyle{plain}
\newtheorem{theorem}{Theorem}[section]
\newtheorem*{theorem*}{Theorem}
\newtheorem{lemma}[theorem]{Lemma}

\newtheorem{prop}[theorem]{Proposition}

\theoremstyle{definition}
\newtheorem{defn}[theorem]{Definition}

\newtheorem{assump}{Assumption}

\newtheorem{example}[theorem]{Example}

\theoremstyle{remark}
\newtheorem*{remark}{Remark}

\DeclareMathOperator{\supp}{supp}

\newcommand{\ie}{\textit{i.e.}}

\newcommand{\ud}{\,\mathrm{d}}
\newcommand{\RR}{\mathbb{R}}
\newcommand{\NN}{\mathbb{N}}

\newcommand{\TT}{\mathrm{T}}

\newcommand{\CC}{\mathbb{C}}
\newcommand{\Or}{\mathcal{O}}

\newcommand{\wt}[1]{\widetilde{#1}}

\newcommand{\mc}[1]{\mathcal{#1}}
\newcommand{\ms}[1]{\mathscr{#1}}

\newcommand{\veps}{\varepsilon}

\newcommand{\abs}[1]{\left\lvert#1\right\rvert}
\newcommand{\norm}[1]{\left\lVert#1\right\rVert}
\newcommand{\average}[1]{\left\langle#1\right\rangle}

\newcommand{\I}{\imath}

\newcommand{\Id}{\mathrm{Id}}

\newcommand{\nn}{\nonumber}
\newcommand{\dsp}{\displaystyle}

\newcommand{\barint}{\kern4pt \raise3.4pt\hbox{\vrule height.6pt
    width7pt} \kern-11pt \int}

\begin{document}

\title[Convergence of frozen Gaussian approximation]{Convergence of
  frozen Gaussian approximation for high frequency wave propagation}

\author{Jianfeng Lu}
\address{Department of Mathematics \\
  Courant Institute of Mathematical Sciences \\
  New York University \\
  New York, NY 10012 \\
  email: jianfeng@cims.nyu.edu }

\author{Xu Yang}
\address{Department of Mathematics \\
  Courant Institute of Mathematical Sciences \\
  New York University \\
  New York, NY 10012 \\
  email: xuyang@cims.nyu.edu }

\date{\today}

\thanks{}

\begin{abstract}
  The frozen Gaussian approximation provides a highly efficient
  computational method for high frequency wave propagation. The
  derivation of the method is based on asymptotic analysis. In
  this paper, for general linear strictly hyperbolic system, we
  establish the rigorous convergence result for frozen Gaussian
  approximation. As a byproduct, higher order frozen Gaussian
  approximation is developed.
\end{abstract}

\maketitle

\section{Introduction}

This paper is devoted to the proof of convergence of the frozen
Gaussian approximation, introduced in \cites{LuYa:CMS,
LuYa:MMS}, for high frequency wave propagation.  Numerical
computation of high frequency wave propagation is an important
problem concerned in many areas, such as seismic imaging,
electromagnetic radiation and scattering, and so on. The problem is
challenging for direct numerical discretization because the mesh
size has to be chosen comparable to the wavelength or even smaller
in order to get accurate solution, however the domain size is
usually large so that the computational cost is formidably
expensive. To look for efficient computation, the development of
asymptotics-based algorithms has received a great amount of
attention in recent years.

The investigations have been focused on two methods: geometric optics
and Gaussian beam method. The computational methods based on geometric
optics (see the review articles \cites{EnRu:03, Ru:07}, and references
therein) solve eikonal and transport equations instead of original
hyperbolic system. This makes the choice of mesh size
frequency-independent, and hence the methods are quite
efficient. However, eikonal equation can develop singularities that
makes the asymptotic approximation break down at caustics. To overcome
this problem, Popov introduced Gaussian beam method in \cite{Po:82},
which constructs solution near geometric rays using Taylor
expansion. Ralston \cite{Ra:82} showed that the method gives a valid
approximation at caustics.  One shortcoming of Gaussian beam method is
however, since it is based on Taylor expansion, the constructed beam
solution has to stay near geometric rays to maintain
accuracy. Therefore the method loses accuracy when the solution
spreads \cites{MoRu:10, QiYi:app1, LuYa:CMS}. The problem is one of
the major concerns in application of Gaussian beam methods in areas
like seismic imaging; see for example \cites{CePoPs:82, Hi:90}.

The frozen Gaussian approximation was proposed in our previous works
\cites{LuYa:CMS, LuYa:MMS} to overcome the problems of the
aforementioned methods: Geometric optics breaks down around
caustics; Gaussian beam method loses accuracy when beam spreads. The
frozen Gaussian approximation is based on asymptotic analysis on
phase plane, motivated by the Herman-Kluk propagator developed in
chemistry literature \cites{HeKl:84, Ka:94, Ka:06}. It provides a
highly efficient computational tool for computing high frequent
solution to linear hyperbolic system.

Our previous work \cites{LuYa:CMS, LuYa:MMS} developed numerical
algorithms for frozen Gaussian approximation in both Lagrangian and
Eulerian framework. Numerical examples indicate the efficiency and
accuracy of the method. In the current work, we prove the
convergence of frozen Gaussian approximation.  Denote the propagator
of the first order strictly hyperbolic system as $\mc{P}_t$, and the
propagator of the $K$th-order frozen Gaussian approximation as
$P_{t, K}^{\veps}$. The main result of this work is the following
theorem.
\begin{theorem*}
  Let $\{u_0^{\veps}\}_{\veps > 0}$ be a family of
  \emph{asymptotically high frequency} initial conditions,  with
  $\norm{u_0^{\veps}}_{L^2(\RR^3)} \leq M$, then we have
  \begin{equation*}
    \norm{\mc{P}_t u_0^{\veps} - P_{t, K}^{\veps} u_0^{\veps}}_{L^2(\RR^3)} \lesssim
    \veps^K M.
  \end{equation*}
\end{theorem*}
Please refer to Section~\ref{sec:form} and Theorem~\ref{thm:main} for
the formulation of the frozen Gaussian approximation and a precise
presentation of the main theorem. The asymptotically high frequency initial
condition is defined in Definition~\ref{def:highfreq}.

\subsection*{Related works}

The frozen Gaussian approximation is motivated by the Herman-Kluk
propagator developed in the chemistry literature \cites{HeKl:84,
  Ka:94, Ka:06}, which is used in the semiclassical regime of time
dependent Schr\"odinger equation. The convergence of Herman-Kluk
propagator was recently proved by Swart and Rousse \cite{SwRo:09} and
Robert \cite{Ro:10}. It is proved that the Herman-Kluk propagator
converges to the true propagator of the Schr\"odinger equation. In
particular, this means that applied to any initial data in $L^2$, the
Herman-Kluk propagator provides an accurate result as $\veps \to 0$.

The difference of the Schr\"odinger equation and the first order
hyperbolic system lies in the fact that the hyperbolic system might
present singularities at $p = 0$ on phase plane (see
Section~\ref{sec:hypersys} for more details). Therefore, one can not
hope to obtain similar results for first order hyperbolic system as
for Schr\"odinger equation.  Indeed, we construct a counter-example
which shows that the method fails to give a good approximation for
low frequency initial data in Example~\ref{example}.  On the other
hand, the results of this work show that frozen Gaussian
approximation works for the initial data that are high frequent.
This is of course the working assumption for high frequency wave
propagation.

The convergence of the Gaussian beam method has been recently
investigated in \cites{LiRa:09, LiRa:10, BoAkAl:09, LiRuTa:10}. The
results in \cites{BoAkAl:09, LiRuTa:10} showed that the $K$-th order
Gaussian beam method converges to the true solution with an accuracy
of $\Or(\veps^{K/2})$. The $K$-th order frozen Gaussian approximation,
as indicated by Theorem~\ref{thm:main}, has a convergence order
$\Or(\veps^K)$. We refer to \cite{LuYa:CMS} for a more detailed
numerical comparison between frozen Gaussian approximation and
Gaussian beam method.

\subsection*{Organization of the paper}

In Section~\ref{sec:prelim}, we introduce some necessary notations
and preliminaries for phase plane analysis. The hyperbolic system we
considered is presented in Section~\ref{sec:hypersys}.
Section~\ref{sec:form} describes the formulation of frozen Gaussian
approximation and states the main convergence theorem. The proof is
based on the construction of high order approximate solution given
in Section~\ref{sec:highorder} and stability from the well-posedness
of hyperbolic system. We conclude the proof in
Section~\ref{sec:proof}.

\section{Preliminaries}\label{sec:prelim}

\subsection{Notations}

We will in general use $x, y \in \RR^d$ as spatial variables, $(q,
p)\in \RR^{2d}$ as the variable for the phase space. $d$ is the
spatial dimensionality. We will use the same notation $\abs{\cdot}$
for absolute value, Euclidean distance, vector norm, (induced) matrix
norm and (induced) tensor norm.

For $\delta > 0$, we define the closed set $K_{\delta} \subset
\RR^{2d}$ as
\begin{equation}\label{def:Kdelta}
  K_{\delta} = \Bigl\{ (q, p) \in \RR^{2d} \Big\vert
  \abs{q} \leq 1 / \delta, \, \abs{p} \in [\delta, 1/\delta] \Bigr\}.
\end{equation}
For $f: \RR^{2d} \to \CC$, we define for $k \in \NN$ and $\delta \in
\RR^+$,
\begin{equation}\label{def:symbol}
  \Lambda_{k, \delta}[f] = \max_{\abs{\alpha} \leq k}
  \sup_{(q, p) \in K_{\delta}} \Bigl\lvert \partial^{\alpha} f(q, p)
  \Bigr\rvert,
\end{equation}
where $\alpha$ is a multi-index. This definition can be extended to
vector valued and matrix valued functions straightforwardly.

For $f: \RR^M \to \CC^N$ and $M,\;N\in\NN$, the matrix
valued function $\partial_x f: \RR^M \to \CC^{M \times N}$ is defined
with the convention that $(\partial_x f(x))_{jk} =
\partial_{x_j} f_k(x)$ for $j = 1, \cdots, M$, $k = 1, \cdots N$.

We use the notations $\mc{S}$, $C^{\infty}$ and $C^{\infty}_c$ for
Schwartz function class, smooth functions and compact supported smooth
functions respectively.

For convenience, we use the notation $\Or(\veps^{\infty})$: $A^{\veps}
= \Or(\veps^{\infty})$ means that for any $k \in \NN$,
\begin{equation*}
  \lim_{\veps \to 0} \veps^{-k} \abs{A^{\veps}} = 0.
\end{equation*}

Notations $c$ and $C$ will be used for constants, whose values might
change from line to line. Sometimes, we will use notations like $C_{T,
  K}$ to specify the dependence of the constant on the parameters $T$ and
$K$.

\subsection{Wave packet decomposition}

For $(q, p) \in \RR^{2d}$, define $\psi^{\veps}_{q,
  p}$ as
\begin{equation}
  \psi^{\veps}_{q, p}(x) = (2\pi\veps)^{-d/2}
  \exp\Bigl( \I p \cdot (x - q) / \veps
  - \tfrac{1}{2} \abs{x - q}^2/\veps\Bigr).
\end{equation}
Define FBI transform $\ms{F}^{\veps}$ on $\mc{S}(\RR^d)$ as
\begin{equation}
  \begin{aligned}
    (\ms{F}^{\veps} f)(q, p)
    & = (\pi \veps)^{-d/4} \langle \psi^{\veps}_{q, p}, f \rangle \\
    & = 2^{-d/2} (\pi\veps)^{-3d/4} \int_{\RR^d} e^{- \I p \cdot (x -
      q) / \veps - \tfrac{1}{2} \abs{x - q}^2/\veps} f(x) \ud x.
  \end{aligned}
\end{equation}
The inverse FBI transform $\bigl(\ms{F}^{\veps}\bigr)^{\ast}$ on $\mc{S}(\RR^{2d})$ is
given by
\begin{equation}
  \bigl((\ms{F}^{\veps})^{\ast} g\bigr)(x) =
  2^{-d/2} (\pi \veps)^{-3d/4} \iint_{\RR^{2d}} e^{\I p \cdot
    (x - q) / \veps - \tfrac{1}{2} \abs{x -
      q}^2/\veps} g(q, p)
  \ud q \ud p.
\end{equation}

We summarize some elementary properties of the FBI transform, whose proof
is standard and can be found, for example, in \cite{Ma:02}.
\begin{prop}\label{prop:isometry}
  For any $f \in \mc{S}(\RR^d)$ and $g \in \mc{S}(\RR^{2d})$,
  \begin{align}
    & \norm{ \ms{F}^{\veps} f }_{L^2(\RR^{2d})} =
    \norm{f}_{L^2(\RR^d)}; \\
    & \norm{ (\ms{F}^{\veps})^{\ast} g }_{L^2(\RR^d)} =
    \norm{g}_{L^2(\RR^{2d})}.
  \end{align}
  Hence, the domain of $\ms{F}^{\veps}$ and $(\ms{F}^{\veps})^{\ast}$
  can be extended to $L^2(\RR^d)$ and $L^2(\RR^{2d})$
  respectively. Moreover, we have
  \begin{equation}
    (\ms{F}^{\veps})^{\ast} \ms{F}^{\veps} = \Id_{L^2(\RR^d)}.
  \end{equation}
\end{prop}

\begin{remark}
  $\ms{F}^{\veps} (\ms{F}^{\veps})^{\ast} \neq
  \Id_{L^2(\RR^{2d})}$.
\end{remark}

\begin{defn}[Asymptotically high frequency function]\label{def:highfreq}
  Let $\{ u^{\veps}\} \subset L^2(\RR^d)$ be a family of functions
  such that $\norm{u^{\veps}}_{L^2}$ is uniformly bounded. Given
  $\delta>0$, we say that $\{u^{\veps}\}$ is \emph{asymptotically high
    frequency with cutoff $\delta$}, if
  \begin{equation*}
    \int_{ \RR^{2d} \backslash K_{\delta} }
    \abs{(\ms{F}^{\veps} u^{\veps})(q, p)}^2 \ud q \ud p
    = \Or(\veps^{\infty})
  \end{equation*}
  as $\veps \to 0$. $K_{\delta}$ is the closed set defined in
  \eqref{def:Kdelta}.
\end{defn}

The definition of asymptotically high frequency functions is motivated
by WKB function, which is typical in study on high frequency
wave propagation.
\begin{example}
  For $A(x)\in C_c^{\infty}(\RR^d)$ and $S(x) \in C^{\infty}(\RR^d)$,
  $\abs{\nabla S(x)} \geq \delta > 0$, the family of WKB functions
  $u^{\veps} = A(x) \exp(\I S(x) /\veps)$ is asymptotically high
  frequency with cutoff $\delta$.
\end{example}

The Definition~\ref{def:highfreq} is also related with the notion of
frequency set, microlocal support in microlocal analysis. Please refer
to \cite{Ma:02} for more details.

\section{Hyperbolic system and Hamiltonian flow}\label{sec:hypersys}

We consider an $N\times N$ linear hyperbolic system in $d$ dimensional
space,
\begin{equation}\label{eq:hypersys}
  \partial_t u + \sum_{l=1}^d A_l(x) \partial_{x_l} u = 0,
\end{equation}
where $u = (u_1, \ldots, u_N)^{\TT} : \RR^d \to \RR^N$ and $A_l: \RR^d
\to \RR^{N\times N},\, 1 \leq l \leq d$ are smooth matrix valued
functions.

We assume that the system \eqref{eq:hypersys} is \emph{strictly
  hyperbolic}, \ie, for any $(q, p) \in \RR^{2d},\, \abs{p} > 0 $, the
matrix $\sum_{l=1}^d p_l A_l(q)$ has $N$ \emph{distinguished} real
eigenvalues, denoted as $\{ H_n(q, p) \}_{n=1}^N$. We denote by
$L_n(q, p)$ and $R_n(q, p)$ the corresponding left and right
eigenvectors, \ie,
\begin{align}\label{eigen:left}
  & \sum_{l=1}^d p_l L_n^{\TT}(q, p) A_l(q) =
  H_n(q, p) L_n^{\TT}(q, p),\\ \label{eigen:right}
  & \sum_{l=1}^d p_l A_l(q) R_n(q, p) =
  H_n(q, p) R_n(q, p),
\end{align}
with the normalization
\begin{equation*}
  L_m^{\TT}(q, p) R_n(q, p) =
  \delta_{mn},
\end{equation*}
where $\delta_{mn}$ is the Kronecker symbol. Note that as $A_l(q)$ is
smooth, $H_n(q, p)$, $R_n(q, p)$ and $L_n(q, p)$ are smooth functions
of $(q, p)$ for $\abs{p} > 0$. Singularities may occur at $p = 0$.

The Hamiltonian flow associated with $H_n(q, p)$ is given for $\abs{p}
> 0$ as
\begin{equation}\label{eq:Hflow}
  \begin{cases}
    \dsp\frac{\ud Q_n(t, q, p)}{\ud t} = \partial_{P_n} H_n\big(Q_n(t,
    q, p), P_n(t, q, p)\big),
    \\[1em]
    \dsp\frac{\ud P_n(t, q, p)}{\ud t} = - \partial_{Q_n} H_n\big(Q_n(t,
    q, p), P_n(t, q, p)\big),
  \end{cases}
\end{equation}
with initial conditions
\begin{equation}\label{ini:Hflow}
  Q_n(0, q, p) = q \quad \text{and} \quad P_n(0, q, p) = p.
\end{equation}
This gives the map $(q, p) \mapsto (Q_n(t, q, p), P_n(t, q, p))$ for
both $t>0$ and $t<0$ (forward and backward flows). For $p = 0$ and $q \in
\RR^d$, we define $Q_n(t, q, p) = q$ and $P_n(t, q, p) = p$.

We make the following assumption for the system
\eqref{eq:hypersys} we consider, which will be assumed for the rest of
the paper without further indication.
\begin{assump}\label{assump:B}
  For each $n = 1, \cdots, N$, there exists constant $C>0$, so that
  the Hamiltonian $H_n$ satisfies for any $(q, p) \in \RR^{2d}$ with
  $\abs{p} > 0$
  \begin{equation*}
    \abs{p \cdot \partial_q H_n(q, p)} \leq C \abs{p}^2,
    \quad \text{and} \quad
    \abs{q \cdot \partial_p H_n(q, p)} \leq C \abs{q}^2.
  \end{equation*}
\end{assump}





Assumption~\ref{assump:B} is understood as a global Lipschitz
condition for the ODEs \eqref{eq:Hflow}, as we can see from the next
Proposition.

\begin{prop}\label{prop:boundaway}
  For each $n = 1, \cdots, N$, $T>0$ and $\delta > 0$, there exists a
  constant $\delta_T>0$ such that $(Q_n(t, q, p), P_n(t, q, p)) \in
  K_{\delta_T}$, \ie, it satisfies
  \begin{equation}\label{eq:boundaway}
    \delta_T \leq \abs{P_n(t, q, p)} \leq 1/\delta_T,
  \end{equation}
  and
  \begin{equation}\label{eq:boundq}
    \abs{Q_n(t, q, p)} \leq 1/ \delta_T,
  \end{equation}
  for any $(q, p) \in K_{\delta}$ and $t \in [0, T]$.
\end{prop}

\begin{proof}
  We suppress the subscript $n$ in the proof, since the argument is the
  same for each branch.

  Differentiating $\abs{P}$ with respect to time and using
  \eqref{eq:Hflow} produce
  \begin{equation*}
    \frac{\ud}{\ud t} \abs{P} = \frac{1}{\abs{P}} P \cdot \frac{\ud}{\ud t}
    P  = - \frac{1}{\abs{P}} P \cdot \partial_{Q} H.
  \end{equation*}
  Hence, by Assumption~\ref{assump:B},
  \begin{equation*}
    \frac{\ud}{\ud t} \abs{P} \leq \frac{1}{\abs{P}} \abs{P \cdot \partial_Q H}
      \leq C \abs{P}.
  \end{equation*}
  Gronwall's inequality implies
  \begin{equation*}
    \max_{0 \leq t \leq T} \abs{P(t, q, p)} \leq C_T \abs{p}.
  \end{equation*}

  On the other hand, differentiating $\abs{P}^{-1}$ gives
  \begin{equation*}
    \frac{\ud}{\ud t} \biggl( \frac{1}{\abs{P}} \biggr)
    = \frac{1}{\abs{P}^3} P \cdot \partial_Q H.
  \end{equation*}
  This yields
  \begin{equation*}
    \frac{\ud}{\ud t} \biggl( \frac{1}{\abs{P}} \biggr)
    \leq \frac{1}{\abs{P}^3} \abs{P \cdot \partial_Q H} \leq
    C \frac{1}{\abs{P}},
  \end{equation*}
  and hence
  \begin{equation*}
    \max_{0 \leq t \leq T} \abs{P(t, q, p)}^{-1}
    \leq C_T \abs{p}^{-1}.
  \end{equation*}

  Therefore, there exist $c_T, C_T > 0$ such that for $ t \in [0, T]$
  and any $q \in \RR^d$,
  \begin{equation*}
    c_T \abs{p} \leq \abs{P(t, q, p)} \leq C_T \abs{p},
  \end{equation*}
  which implies \eqref{eq:boundaway}.

  The time derivative of $\abs{Q}$ can be bounded by
  \begin{equation*}
    \frac{\ud}{\ud t} \abs{Q} = \frac{1}{\abs{Q}} Q \cdot \partial_P H
    \leq C\abs{Q}
  \end{equation*}
  by Assumption~\ref{assump:B}. Equation \eqref{eq:boundq} then
  follows from Gronwall's inequality.
\end{proof}

\begin{prop}\label{prop:Hbound}
  For $l = 1, \cdots, d$, $k \in \NN$ and $\delta > 0$, there exists
  constant $C_{k, \delta}$ such that
  \begin{equation*}
    \Lambda_{k, \delta}[A_l] \leq C_{k, \delta}.
  \end{equation*}
  Moreover, for each $n = 1, \cdots, N$, there exists constant $C_{k,
    \delta}$ such that
  \begin{equation*}
    \Lambda_{k, \delta}[H_n] \leq C_{k, \delta}, \quad
    \Lambda_{k, \delta}[R_n] \leq C_{k, \delta}, \quad
    \Lambda_{k, \delta}[L_n] \leq C_{k, \delta}.
  \end{equation*}
\end{prop}

\begin{proof}
  The conclusion follows immediately from the smoothness of $A_l, H_n,
  R_n, L_n$ for $\abs{p} > 0$ and the compactness of the set
  $K_{\delta}$.
\end{proof}

We define canonical transformation and action associated with
Hamiltonian flow.

\begin{defn}[Canonical transformation]
  Let
  \begin{equation}\label{eq:kappa}
    \kappa:\quad\begin{aligned}
      \RR^{2d} &\to \RR^{2d} \\ (q,p) &\mapsto \big(Q^{\kappa}(q, p),
      P^{\kappa}(q, p)\big) \end{aligned}
  \end{equation}
  be differentiable for $\abs{p} > 0$. We denote the Jacobian matrix
  as
  \begin{equation}
    F^{\kappa}(q, p) =
    \begin{pmatrix}
      (\partial_q Q^{\kappa})^{\TT}(q, p) & (\partial_p
      Q^{\kappa})^{\TT}(q, p)
      \\
      (\partial_q P^{\kappa})^{\TT}(q, p) & (\partial_p
      P^{\kappa})^{\TT}(q, p)
    \end{pmatrix}.
  \end{equation}
  We say $\kappa$ is a \emph{canonical transformation} if $F^{\kappa}$
  is symplectic for any $(q, p) \in \RR^{2d}$ with $\abs{p} > 0$,
  \begin{equation}
    (F^{\kappa})^{\TT}
    \begin{pmatrix}
      0 & \Id_d \\
      -\Id_d & 0
    \end{pmatrix}
    F^{\kappa} =
    \begin{pmatrix}
      0 & \Id_d \\
      -\Id_d & 0
    \end{pmatrix}.
  \end{equation}
\end{defn}

As a corollary to Proposition~\ref{prop:Hbound}, the map $\kappa_{n,
  t}: (q, p) \mapsto (Q_n(t, q, p), P_n(t, q, p)) $ is a canonical
transformation. We formulate this as the next proposition, which also
gives additional smoothness properties for the Jacobian.

\begin{prop}\label{prop:Jacobian}
  For each $n = 1, \cdots, N$, the map $\kappa_{n,t}$ is a canonical
  transformation, and for any $k\geq 0$, $\delta > 0$ and $T>0$, there
  exists constant $C_{k,\delta, T}$ such that
  \begin{equation}\label{eq:Fbound}
    \sup_{t \in [0, T]} \Lambda_{k,
      \delta}[F^{\kappa_{n,t}}] \leq C_{k, \delta, T}.
  \end{equation}
\end{prop}

\begin{proof}
  The argument is the same for different branches, hence we will
  suppress the subscript $n$ for simplicity.

  Proposition~\ref{prop:boundaway} shows that there exists
  $\delta_T$ such that $\bigl(Q(t, q, p), P(t, q, p)\bigr) \in K_{\delta_T}$ for
  $(q, p) \in K_{\delta}$ and $ t\in [0, T]$.

  Differentiating $F^{\kappa_t}(q, p)$ with respect to time gives
  \begin{equation}\label{eq:evodtf}
    \frac{\ud}{\ud t} F^{\kappa_t}(q, p)
    =
    \begin{pmatrix}
      \partial_P \partial_Q H & \partial_P \partial_P H \\
      - \partial_Q \partial_Q H & - \partial_Q \partial_P H
    \end{pmatrix}
    F^{\kappa_t}(q, p).
  \end{equation}
  Proposition~\ref{prop:boundaway} and \ref{prop:Hbound} imply
  \begin{equation*}
    \frac{\ud}{\ud t} \abs{F^{\kappa_t}(q, p)}
    =
    \left\lvert
    \begin{pmatrix}
      \partial_P \partial_Q H & \partial_P \partial_P H \\
      - \partial_Q \partial_Q H & - \partial_Q \partial_P H
    \end{pmatrix}\right\rvert
  \abs{F^{\kappa_t}(q, p)} \leq C \abs{F^{\kappa_t}(q, p)},
  \end{equation*}
  where $C$ is independent of $(q, p)$. Hence, by Gronwall's
  inequality and $F^{\kappa_0}(q, p) = \Id_{2d}$, one has
  \begin{equation*}
    \abs{F^{\kappa_t}(q, p)} \leq \exp(C \abs{t}).
  \end{equation*}
  Differentiating \eqref{eq:evodtf} with $(q, p)$ yields
  \begin{equation*}
    \frac{\ud}{\ud t} \partial^{\alpha}_{(q, p)} F^{\kappa_t}(q, p)
    =
    \sum_{\beta \leq \alpha}
    {\alpha \choose \beta}
    \partial^{\beta}_{(q, p)}
    \begin{pmatrix}
      \partial_P \partial_Q H & \partial_P \partial_P H \\
      - \partial_Q \partial_Q H & - \partial_Q \partial_P H
    \end{pmatrix}
    \partial^{\alpha-\beta}_{(q, p)} F^{\kappa_t}(q, p).
  \end{equation*}
  The estimate \eqref{eq:Fbound} follows by an induction argument.
\end{proof}

\begin{defn}[Action] Let $\kappa$ be a canonical transformation
  defined in \eqref{eq:kappa}. A function $S^{\kappa}: \RR^{2d} \to
  \RR$ is called an \emph{action associated to $\kappa$} if it
  satisfies
  \begin{align}
    \partial_p S^{\kappa}(q, p) & = (\partial_p Q^{\kappa}(q, p))
    P^{\kappa}(q, p), \\
    \partial_q S^{\kappa}(q, p) & = - p + (\partial_q Q^{\kappa}(q,
    p)) P^{\kappa}(q, p),
  \end{align}
  for any $\abs{p} > 0$.
\end{defn}

The action $S_n(t, q, p) = S^{\kappa_{n,t}}(q, p)$ corresponding to
$H_n(q, p)$ solves the following equation
\begin{equation}\label{eq:S}
  \frac{\ud S_n}{\ud t} = P_n \cdot\partial_P H_n(P_n, Q_n) - H_n(P_n, Q_n),
\end{equation}
with initial condition $S_n(t,q,p)=0$. It is easy to check that $S_n$
is indeed the action associated with $\kappa_{n,t}$, which is given by
\begin{equation}
  S_n(t, q, p) =
  \int_0^t P_n(\tau, q, p)\cdot
  \partial_{\tau} Q_n(\tau, q, p) - H_n
  \bigl(Q_n(\tau, q, p), P_n(\tau, q, p)\bigr) \ud \tau.
\end{equation}

Now we are ready to construct the Fourier integral operator which will
be used in the definition of frozen Gaussian approximation in the
next section.

\begin{defn}[Fourier Integral Operator]
  For $M \in L^{\infty}(\RR^{2d}; \CC^{N \times N})$, a Schwartz-class
  function $u \in \mc{S}(\RR^d; \CC^N)$ and $n = 1, \cdots, N$, we
  define
  \begin{equation}\label{eq:FIO}
    (\mc{I}^{\veps}_n(t, M) u)(x) =
    (2\pi\veps)^{-3d/2} \int_{\RR^{3d}}
    e^{\I \Phi_n(t, x, y, q, p)/\veps} M(q, p) u(y)
    \ud q \ud p \ud y.
  \end{equation}
  where the phase function $\Phi_n$ is given by
  \begin{multline}\label{eq:phim}
    \Phi_n(t, x, y, q, p) = S_n(t, q, p) + \frac{\I}{2} \abs{x -
      Q_n(t, q, p)}^2 + P_n\cdot(x - Q_n(t, q, p)) \\
    + \frac{\I}{2} \abs{y - q}^2 - p\cdot(y - q).
  \end{multline}
\end{defn}

\begin{prop}\label{prop:L2bound}
  If $M \in L^{\infty}(\RR^{2d}; \CC^{N\times N})$, for each $n = 1,
  \cdots, N$ and any $t$, $\mc{I}^{\veps}_n(t, M)$ can be extended to
  a linear bounded operator on $L^2(\RR^d; \CC^N)$, and we have
  \begin{equation}
    \norm{\mc{I}^{\veps}_n(t, M)}_{\ms{L}(L^2(\RR^d; \CC^N))} \leq 2^{-d/2}
    \norm{M}_{L^{\infty}(\RR^{2d}; \CC^{N\times N})}.
  \end{equation}
\end{prop}

\begin{proof}
  For $u, v \in L^2(\RR^d)$, taking the inner product of $v$ and
  \eqref{eq:FIO} gives
  \begin{equation*}
    \begin{aligned}
      \average{v, \mc{I}^{\veps}_n(t, M)u} & = (2\pi\veps)^{-3d/2}
      \int_{\RR^{4d}} v(x)^{\ast} e^{\I \Phi_n(t, x, y, q, p)/\veps}
      M(q, p) u(y) \ud q \ud p \ud y \ud x \\
      & = (2\pi\veps)^{-3d/2} \int_{\RR^{2d}} \ud q \ud p \; e^{\I
        S_n(t, q, p)/\veps} \\
      & \hspace{7em} \times \overline{\biggl(\int_{\RR^d} e^{-
          \tfrac{\I}{\veps} P_n\cdot(x - Q_n) -
          \tfrac{1}{2\veps}\abs{x - Q_n}^2} v(x) \ud
        x\biggr)} \\
      & \hspace{7em} \times M(q, p) \int_{\RR^d} e^{-\tfrac{\I}{\veps}
        p\cdot (y
        -q) - \tfrac{1}{2\veps}\abs{y - q}^2} u(y) \ud y \\
      & = 2^{-d/2}\int_{\RR^{2d}} e^{\I S_n(t, q, p)/\veps}
      \overline{(\ms{F}^{\veps} v)}(Q_n, P_n) M(q, p)(\ms{F}^{\veps}
      u)(q, p) \ud q \ud p.
    \end{aligned}
  \end{equation*}
  Therefore,
  \begin{multline*}
    \abs{\average{v, \mc{I}^{\veps}_n(t, M)u}} \leq 2^{-d/2}\norm{
      (\ms{F}^{\veps} v) \circ \kappa_{n,t} }_{L^2(\RR^{2d}; \CC^N)}
    \norm{ e^{\I S_n/\veps} M (\ms{F}^\veps u) }_{L^2(\RR^{2d};
      \CC^N)}
    \\
    \leq 2^{-d/2}\norm{v}_{L^2(\RR^d; \CC^N)}
    \norm{M}_{L^{\infty}(\RR^{2d}; \CC^{N\times N})}
    \norm{u}_{L^2(\RR^d; \CC^N)},
  \end{multline*}
  where we have used Proposition~\ref{prop:isometry} and the
  symplecticity of the canonical transform $\kappa_{n,t}$. This
  implies
  \begin{equation*}
    \norm{\mc{I}^{\veps}_n(t, M)u}_{L^2(\RR^d; \CC^N)} \leq 2^{-d/2}
    \norm{M}_{L^{\infty}(\RR^{2d}; \CC^{N\times N})}
    \norm{u}_{L^2(\RR^d; \CC^N)},
  \end{equation*}
  which proves the Proposition.
\end{proof}

\section{Formulation and main result}\label{sec:form}

Let $\chi_{\delta}: \RR^{2d} \to [0, 1]$ be a smooth cutoff function
for $\delta > 0$ so that
\begin{equation*}
  \chi_{\delta}(q, p) =
  \begin{cases}
    1, & (q, p) \in K_{\delta};\\
    0, & (q, p) \in \RR^{2d} \backslash K_{\delta/2},
  \end{cases}
\end{equation*}
and for any $k \in N$, there exists constant $C_{k, \delta}$ such that
\begin{equation*}
  \sup_{(q, p) \in \RR^{2d}} \sup_{\abs{\alpha} = k}
  \abs{\partial^{\alpha}_{(q, p)} \chi_{\delta}(q, p)} \leq C_{k, \delta}.
\end{equation*}
The construction of the cutoff function is standard.

Construct the $K$-th order frozen Gaussian approximation with the cutoff
$\delta$ as
\begin{equation}\label{eq:ansatz}
  \begin{aligned}
    (\mc{P}_{t,K, \delta}^{\veps} u_0)(x) & = \sum_{n=1}^N
    \left(\mc{I}_n^{\veps}\bigl(t, \mc{M}^{\veps}_{n, K, t}
      \chi_{\delta} \bigr)
      u_0\right)(x) \\
    & = \frac{1}{(2\pi \veps)^{3d/2}} \sum_{n=1}^N \int_{\RR^{3d}}
    e^{\I\Phi_n/\veps} \mc{M}^{\veps}_{n,K, t}(q, p) \chi_{\delta}(q, p)
    u_{0}(y) \ud q \ud p \ud y.
\end{aligned}
\end{equation}
Here, the matrix symbol $\mc{M}^{\veps}_{n, K, t}$ is given by
\begin{equation}
  \mc{M}^{\veps}_{n, K, t}(q, p) =
  \mc{M}^{\veps}_{n, K}(t, q, p) = \sum_{k=0}^{K-1}\veps^k M_{n,k}(t, q, p).
\end{equation}
Before we introduce the $N\times N$ matrix valued function
$M_{n,k}(t, q, p)$, we need to define some operators first.  We denote
\begin{equation}
  \label{eq:op_zZ}
  \partial_{z}=\partial_{q}-\I\partial_{p}, \qquad
  Z_n(t, q, p)=\partial_{z}\bigl(Q_n(t, q, p)+\I P_n(t, q, p)\bigr).
\end{equation}
Note that $Z_n(t,q,p)$ is invertible for $\abs{p} > 0$, which will be
proved in Lemma~\ref{lem:Zinv}. Without further indication, we are
going to use the Einstein summation convention, except for the branch
index $n$.

We denote $b: \RR^{2d} \to \CC$ as a generic function defined on
$\RR^{2d}$. For indices $j_1, j_2 \in \{1, \cdots, d\}$, define
operators $\mc{D}$ and $\mc{G}$ as
\begin{align}
  \label{op:T_j} & (\mc{D}_{n, j_1} b)(q, p)
  = - \partial_{z_l} \bigl(b(q, p) Z^{-1}_{n, j_1 l}(t, q, p)\bigr), \\
  \label{op:T_j1j2} & (\mc{G}_{n, j_1j_2} b)(q, p) = \partial_{z_l}
  Q_{n, j_1}(t, q, p) b(q, p) Z^{-1}_{n, j_2 l}(t, q, p),
\end{align}
provided the right hand sides are well defined (in particular, $b$ is
differentiable at $(q, p)$).

For $\nu \in \NN$, $j_{\alpha} \in \{1,
\cdots, d\}$ for $\alpha = 1, \cdots, \nu$, and $\veps > 0$, define
operators $\mc{T}_{n, j_1\cdots j_{\nu}}^{\nu, \veps}$ as
\begin{align}\label{op:T1}
  & \mc{T}^{1,\veps}_{n, j_1} b = \veps \mc{D}_{n, j_1} b, \\
  & \mc{T}^{2,\veps}_{n, j_1j_2} b = \veps \mc{G}_{n, j_1j_2} b +
  \veps^2 \mc{D}_{n,j_1}\mc{D}_{n,j_2} b, \\
  \shortintertext{and for $\nu \geq 3$,}
  \label{op:Tnu}
  & \mc{T}^{\nu,\veps}_{n, j_1\cdots j_{\nu}} b =
  \veps\mc{T}^{\nu-1,\veps}_{n, j_1\cdots j_{\nu-1}}\mc{D}_{n,
    j_{\nu}} b +\veps\sum_{\alpha=1}^{\nu-1} \mc{T}^{\nu-2,\veps}_{n,
    j_1\cdots j_{\alpha-1}j_{\alpha+1}\cdots j_{\nu-1}}\mc{G}_{n,
    j_{\alpha}j_{\nu}} b,
\end{align}
provided the right hand sides are well defined at $(q, p) \in
\RR^{2d}$. Notice that $\mc{D}$, $\mc{G}$ and $\mc{T}$ depend on $t$
as well, which we choose not to make explicit in notation for
simplicity, with the hope that no confusion will occur.

It is easy to see that we can rewrite $\mc{T}^{\nu,\veps}_{n, j_1\cdots
  j_{\nu}}$ in orders of $\veps$ as
\begin{equation}
  \mc{T}^{\nu,\veps}_{n, j_1\cdots j_{\nu}} b = \sum_{k=\lceil {\nu}/2 \rceil}^{\nu}
  \veps^{k} \mc{T}^{\nu,k}_{n, j_1 \cdots j_{\nu}} b,
\end{equation}
where $\lceil \cdot \rceil$ is the ceiling function.  The last
equality defines $\mc{T}^{\nu,k}_{n, j_1\cdots j_{\nu}}$.  The
definition of these operators $\mc{D}$, $\mc{G}$ and $\mc{T}$ can be
extended to vector-valued and matrix-valued functions so that the
operators act on the function elementwisely. For example, for the matrix-valued function $M:
\RR^{2d} \to \CC^{N\times N}$, we have
\begin{equation*}
  (\mc{D}_{n, j_1} M)_{kl} = \mc{D}_{n, j_1} M_{kl},
\end{equation*}
for $k, l = 1, \cdots, N$.

We define the operator $\mc{L}_{n, k}$ as follows. For $M: \RR \times
\RR^{2d} \to \CC^{N \times N}$ and $(t, q, p) \in \RR \times \RR^{2d}$
with $\abs{p} > 0$, $\mc{L}_{n, k}(M)$ is given by
\begin{equation}\label{op:L0}
  \begin{aligned}
    ( \mc{L}_{n, 0} M)(t, q, p) & =\I\Bigl(\partial_t S_n -P_{n,j}
    \partial_t Q_{n, j} + P_{n, j} A_j(Q_n) \Bigr) M(t, q, p) \\
    & = \I \Bigl( P_{n, j} A_j(Q_n) - H_n(Q_n, P_n) \Id_N \Bigr) M(t,
    q, p),
  \end{aligned}
\end{equation}
\begin{equation}\label{op:L1}
  \begin{aligned}
    (\mc{L}_{n, 1}M)(t, q, p) & =\partial_t M(t, q, p) \\
    & +\I \mc{D}_{n, j}\biggl(\Bigl((\partial_t P_{n,j}-\I\partial_t
    Q_{n, j})\Id_N \\
    & \qquad + \I
    A_j(Q_n) + P_{n,l}\partial_{j}A_l(Q_n)\Big) M(t, q, p)\biggr) \\
    & +\I \mc{G}_{n, j_1j_2}\biggl( \Bigl(\I\partial_{Q_{n,
        j_1}} A_{j_2}(Q_n) \\
    & \qquad + \frac{1}{2}
    P_{n,l}\partial^2_{j_1j_2}A_l(Q_n)\Bigr) M(t, q,
    p) \biggr),
  \end{aligned}
\end{equation}
where $S_n, Q_n$ and $P_n$ on the right hand sides are evaluated at
$(t, q, p)$ and we have introduced the short-hand,
\begin{equation*}
  \partial^{\nu}_{j_1\cdots j_{\nu}}A_l(Q_n)=
  \partial^{\nu}_{Q_{n,j_1}\cdots
    Q_{n,j_{\nu}}}A_l(Q_n), \quad\hbox{for any $\nu\geq 1$}.
\end{equation*}
Moreover, for $k \geq 2$, we define
\begin{equation}
  \begin{aligned}
    (\mc{L}_{n, k} M)(t, q, p) & =\sum_{\nu=k}^{2k}\mc{T}_{n,
      j_1\cdots j_{\nu}}^{\nu,k} \biggl(\Bigl(
    -\frac{1}{(\nu-1)!}\partial^{\nu-1}_{j_1\cdots
        j_{\nu-1}}A_{j_\nu}(Q_n)\\
    & \qquad +\frac{\I}{\nu!}  P_{n,l}\partial^{\nu}_{j_1\cdots
      j_{\nu}}A_l(Q_n)\Bigr) M(t, q, p)\biggr).
  \end{aligned}
\end{equation}
We remind the readers that $\mc{D}$, $\mc{G}$ and $\mc{T}$ depend on
$t$ implicitly.

We are ready to give the matrix-valued function $M_{n,k}$ now.
$M_{n,k}$ is defined for $ t \in [0, T]$ and $(q, p) \in \RR^{2d}$
with $\abs{p} > 0$. First, $M_{n,0}$ is given by
\begin{equation}
  M_{n,0}(t,q,p)=\sigma_{n,0}(t,q,p)
  R_n(Q_n(t, q, p) ,P_n(t, q, p))L_n(q,p)^{\TT},
\end{equation}
where $\sigma_{n,0}$ is determined by the evolution equation,
\begin{equation}\label{eq:sigma}
  \frac{\ud}{\ud t} \sigma_{n,0}(t, q, p) +
  \sigma_{n,0}(t, q, p)\lambda_n(t, q, p) = 0,
\end{equation}
with the initial condition
\begin{equation}\label{ini:sigma_la}
\sigma_{n,0}(0,q, p)=2^{d/2}.
\end{equation}
In \eqref{eq:sigma}, we have used the short-hand
\begin{equation}\label{eq:lambda}
  \begin{aligned}
    \lambda_n(t,q, p) & = L_n^{\TT} \bigl(\partial_{P_n} H_n \cdot
    \partial_{Q_n} R_n - \partial_{Q_n} H_n \cdot \partial_{P_n} R_n\bigr) \\
    &  - (\partial_{z_k} L_n)^{\TT} \Bigl( (A_j - \partial_{P_{n,j}}
    H_n) + \I ( \partial_{Q_{n,j}} H_n
    - P_{n,l} \partial_{j} A_l ) \Bigr) R_n Z_{n,jk}^{-1}  \\
    &  + \partial_{z_s} Q_{n,j} Z_{n,ks}^{-1} L_n^{\TT} \bigl( -
    \partial_{j} A_k + \frac{\I}{2} P_{n,l}
    \partial^2_{jk} A_l \bigr) R_n,
  \end{aligned}
\end{equation}
where $Q_n, P_n$ are evaluated at $(t, q, p)$, $A_j$'s are evaluated
at $Q_n$, and $H_n, L_n, R_n$ are evaluated at $(Q_n, P_n)$.

Notice that the action $\mc{L}_{n,0}$ is just multiplication with the
matrix $ \I \bigl(P_{n,j} A_j(Q_n) - H_n(Q_n, P_n) \Id_N \bigr)$ on
the left. We define the matrix,
\begin{equation}\label{eq:L0dagger}
  \mc{L}_{n,0}^{\dagger}(q, p) = \I \sum_{m\neq n}
  \bigl(H_n(q, p) - H_m(q, p)\bigr)^{-1}
  R_m(q, p) L_m^{\TT}(q, p).
\end{equation}
For $k \geq 1$, $M_{n,k}$ is given by
\begin{multline}\label{eq:a_s}
  M_{n,k}(t,q, p)=\sigma_{n,k}(t,q, p)R_n\bigl(Q_n(t, q, p), P_n(t, q,
  p)\bigr) L_n(q, p)^{\TT} \\
  + M_{n,k}^\perp(t,q, p),
\end{multline}
where $M_{n,k}^\perp(t,q, p)$ is given by
\begin{equation}\label{eq:a_sperp}
  M_{n,k}^\perp(t,q,p)=\mc{L}_{n, 0}^{\dagger}\bigl(Q_n(t,q,p),P_n(t, q, p)\bigr)
  \sum_{s=1}^k(\mc{L}_{n,s}
  M_{n,k-s})(t,q,p),
\end{equation}
and $\sigma_{n,k}(t,q, p)$ solves
\begin{multline}\label{eq:sigma_s}
  \frac{\ud \sigma_{n,k}}{\ud t} +
  \sigma_{n,k}\lambda_n+L_n\bigl(Q_n(t,q,p),
  P_n(t,q,p)\bigr)^{\TT}\bigg((\mc{L}_{n,1} M_{n,k}^\perp )(t, q, p)\\
  +\sum_{s=2}^{k+1} (\mc{L}_{n,s} M_{n,k+1-s})(t, q, p) \bigg)R_n(q,
  p)=0,
\end{multline}
with the initial condition $ \sigma_{n,k}(0,q, p)=0$.

We now state the main result.  The following theorem indicates that
the $K$-th order frozen Gaussian approximation (FGA) gives an order
$\Or(\veps^K)$ approximate solution to strictly linear hyperbolic
system.

\begin{theorem}\label{thm:main}
  Consider a strictly linear hyperbolic system \eqref{eq:hypersys}
  under Assumption~\ref{assump:B}. For a family
  of initial conditions $\{u_0^{\veps}\}$ that is asymptotically high
  frequency with cutoff $\delta$, and $\norm{u_0^{\veps}}_{L^2(\RR^d)} \leq M$,
  then for any $T>0$ and $K\in \NN$, there exist constants $C_{T, K}$ and
  $\veps_0> 0$, such that for $\veps \in (0, \veps_0]$,
  \begin{equation*}
    \max_{0\leq t \leq T} \left\lVert \mc{P}_t u_0^{\veps} - \mc{P}_{t, K, \delta}^{\veps}
      u_0^{\veps}
    \right\rVert_{\ms{L}(L^2(\RR^d))}
    \leq C_{T,K} M \veps^{K}.
  \end{equation*}
\end{theorem}

\begin{remark}
  The cutoff $\delta$ is used in the formulation to avoid
  singularities presented at $\abs{p} = 0$. From an numerical point
  of view, the cutoff is quite natural. Indeed, in the numerical implementation of
  frozen Gaussian approximation \cites{LuYa:CMS, LuYa:MMS}, a
  cutoff on phase plane is always used for efficiency.
\end{remark}

\begin{remark}
  The necessity of assumption that the initial value is asymptotically
  high frequency is presented in the following example. It shows that the
  FGA does not work if the FBI transform of the initial condition
  concentrates around $\abs{p} = 0$ as $\veps \to 0$.
\end{remark}

\begin{example}\label{example}
  Consider the acoustic wave equation in two dimension with constant
  coefficients,
  \begin{equation*}
    \partial_t {u} + A_1 \partial_{x_1} {u}
    + A_2 \partial_{x_2} {u} = 0,
  \end{equation*}
  where
  \begin{equation*}
    A_1 =
    \begin{pmatrix}
      0 & 0 & 1 \\
      0 & 0 & 0 \\
      1 & 0 & 0
    \end{pmatrix},
    \quad
    A_2 =
    \begin{pmatrix}
      0 & 0 & 0 \\
      0 & 0 & 1 \\
      0 & 1 & 0
    \end{pmatrix}.
  \end{equation*}

  Then the eigenfunctions in \eqref{eigen:left}-\eqref{eigen:right} are given by
  \begin{equation*}
    H_1({q}, {p}) = 0, \quad
    H_{\pm}({q}, {p}) = \pm \abs{{p}},
  \end{equation*}
  where we have used the subscripts $\pm$ instead of number
  subscripts. This implies the system is strictly hyperbolic, and the
  corresponding eigenvectors are
  \begin{equation*}
    \begin{aligned}
      & {R}_1({q}, {p}) =
      \begin{pmatrix}
        p_2  \\
        - p_1 \\
        0
      \end{pmatrix}, \quad {R}_{\pm}({q}, {p}) =
      \begin{pmatrix}
        \pm p_1  \\
        \pm p_2  \\
        \abs{{p}}
      \end{pmatrix}, \\
      & {L}_1({q}, {p}) = \frac{1}{\abs{{p}}^2}
      \begin{pmatrix}
        p_2 \\
        - p_1 \\
        0
      \end{pmatrix}, \quad {L}_{\pm}({q}, {p}) = \frac{1}{2}
      \begin{pmatrix}
        \pm p_1/\abs{{p}}^2 \\
        \pm p_2 /\abs{{p}}^2 \\
        1/ \abs{{p}}
      \end{pmatrix}.
    \end{aligned}
  \end{equation*}

  It is easy to see that the system satisfies Assumption
  \ref{assump:B}.

  We focus on the ``+'' branch and omit the subscript ``+'' for
  convenience, then the Hamiltonian flow is given by
  \begin{equation}
    \begin{cases}
      \dsp\frac{\ud Q}{\ud t} = \frac{P}{|P|},
      \\[1em]
      \dsp\frac{\ud P}{\ud t} = 0,
    \end{cases}
  \end{equation}
  with the initial conditions $Q(0,q,p)=q$ and $P(0,q,p)=p$, which implies
  the following solution,
  \begin{equation}
    Q=q+\frac{p}{|p|}t,\quad P=p.
  \end{equation}
  Therefore
  \begin{equation}
    Z=2\Id_2 -\frac{\I t}{|p|}\bigg(\Id_2-\frac{p\otimes p}{|p|^2}\bigg),
    \quad \det Z=4-\frac{2\I t}{|p|}.
  \end{equation}
  According to the equation $(2.7)$ in \cite{LuYa:CMS}, the solution
  of amplitude $\sigma(t,q,p)$ is
  \begin{equation}
    \sigma=\bigg( 4-\frac{2\I t}{|p|} \bigg)^{1/2},
  \end{equation}
  with the branch of square root determined continuously in time by
  initial condition.

  If we take the initial condition of $u$ as
  \begin{equation}\label{ini:c_ex}
    u_0(x)=(u_{0,1},u_{0,2},u_{0,3})^\TT=\bigl(\exp(-|x|^2/2),0,0\bigr)^\TT,
  \end{equation}
  then the FBI transform gives
  \begin{equation}
    \begin{aligned}
      & \ms{F}^\veps(u_{0,1})=\frac{1}{(\pi \veps)^{1/2}(1+\veps)}
      \exp\bigg(\frac{ip\cdot q}{1+\veps}
      -\frac{|q|^2}{2(1+\veps)}-\frac{|p|^2}{2\veps(1+\veps)} \bigg), \\
      & \ms{F}^\veps(u_{0,2})=0, \\
      & \ms{F}^\veps(u_{0,3})=0.
    \end{aligned}
  \end{equation}

  The leading order frozen Gaussian approximation of $u$ is given by
  \begin{equation}
    u^{\veps}(t,x)=\frac{1}{2}(\ms{F})^*\bigl( \sigma(t,q,p) R(Q,P)L^{\TT}(q,p)
    \ms{F}^\veps(u_0) \bigr),
  \end{equation}
  By Proposition \ref{prop:isometry},
  \begin{equation}
    \begin{aligned}
      \norm{u^{\veps}}_{L^2}^2 &= \norm{\frac{1}{2}\sigma(t,q,p)
        R(Q,P)L^{\TT}(q,p) \ms{F}^\veps(u_0)}_{L^2}^2 \\
      &=\frac{1}{(\pi \veps)(1+\veps)^2}\int_{\RR^4}\abs{1-\frac{\I
          t}{2|p|}}\frac{p_1^4}{4|p|^4}
      \exp\bigg({-\frac{|q|^2}{1+\veps}-\frac{|p|^2}{\veps(1+\veps)}}\bigg)\ud
      q \ud p.
    \end{aligned}
  \end{equation}
  Notice that, as $\veps \to 0$,
  \begin{equation*}
    \int_{\RR^2}e^{-\frac{|q|^2}{1+\veps}}\ud q = \Or(1),
  \end{equation*}
  and
  \begin{equation*}
    \begin{aligned}
      \frac{1}{\veps}\int_{\RR^2}\abs{1-\frac{\I
          t}{2|p|}}\frac{p_1^4}{4|p|^4}
      e^{-\frac{|p|^2}{\veps(1+\veps)}}\ud p
      &=\frac{1}{\veps}\int_0^\infty\int_0^{2\pi}\abs{1-\frac{\I
          t}{2r}}\frac{\cos^4\theta}{4}
      e^{-\frac{r^2}{\veps(1+\veps)}}r\ud r \ud \theta, \\
      &=\frac{1}{\sqrt{\veps}}\int_0^\infty\int_0^{2\pi}\abs{\sqrt{\veps}r-\frac{\I
          t}{2}}\frac{\cos^4\theta}{4}
      e^{-\frac{r^2}{(1+\veps)}}\ud r \ud \theta\\
      & = \Or(\veps^{-1/2}),
    \end{aligned}
  \end{equation*}
  where the second equality is obtained by change of variable.

  Hence one has $\norm{u^{\veps}}_{L^2}$ is of the order
  $\Or(\veps^{-1/4})$, while $\norm{u}_{L^2}$ under the initial
  condition \eqref{ini:c_ex} is of the order $\Or(1)$, which implies
  FGA can not be a good approximation to the acoustic wave equation under
  this choice of initial condition.

\end{example}

\section{High order approximation}\label{sec:highorder}

In this section we introduce a high order approximation to the
solution of \eqref{eq:hypersys} based on frozen Gaussian
approximation.  This is a key step in the proof of
Theorem~\ref{thm:main}. We state and prove some preliminary lemmas
first.

For a canonical transformation $\kappa_{n,t}$, define
$Z^{\kappa_{n,t}}{(q,p)}$ for $\abs{p} > 0$ as
\begin{equation*}
  Z^{\kappa_{n,t}}(q, p) = \partial_z\bigl(Q^{\kappa_{n,t}}{(q,p)}
+ \I P^{\kappa_{n,t}}{(q,p)}\bigr),
\end{equation*}
which is related to $Z_n(t, q, p)$ defined in \eqref{eq:op_zZ} by
\begin{equation*}
  Z_n(t, q, p) = Z^{\kappa_{n, t}}(q, p).
\end{equation*}

\begin{lemma}\label{lem:Zinv}
  $Z^{\kappa_{n,t}}{(q,p)}$ is invertible for each $n=1,\cdots,N$ and
  $(q, p)\in \RR^{2d}$ with $\abs{p} > 0$.  Moreover, for any $k\geq
  0$ and $\delta > 0$, there exists constant $C_{k, \delta}$ such that
  \begin{equation}\label{eq:invZbound}
    \Lambda_{k, {\delta}}[(Z^{\kappa_{n,t}})^{-1}{(q,p)}] \leq C_{k, \delta}.
  \end{equation}
\end{lemma}

\begin{proof}
  Since the proof is the same for each branch, we omit the subscript
  $n$ in notation for convenience.

  Observe that $Z^{\kappa_t}{(q,p)}$ can be rewritten as
  \begin{equation*}
    Z^{\kappa_t}{(q,p)} = \partial_{z}\bigl(Q^{\kappa_t}{(q,p)}
    + \I P^{\kappa_t}{(q,p)} \bigr)
    =
    \begin{pmatrix}
      \I \Id_d & \Id_d
    \end{pmatrix}
    (F^{\kappa_t})^{\TT}{(q,p)}
    \begin{pmatrix}
      - \I \Id_d \\
      \Id_d
    \end{pmatrix}.
  \end{equation*}
  Therefore,
  \begin{equation*}
    \begin{aligned}
      \bigl(Z^{\kappa_t} (Z^{\kappa_t})^{\ast}\bigr){(q,p)} & =
      \begin{pmatrix}
        \I \Id_d & \Id_d
      \end{pmatrix}
      (F^{\kappa_t})^{\TT}{(q,p)}
      \begin{pmatrix}
        \Id_d  & -\I \Id_d \\
        \I \Id_d & \Id_d
      \end{pmatrix}
      F^{\kappa_t}{(q,p)}
      \begin{pmatrix}
        -\I \Id_d \\
        \Id_d
      \end{pmatrix} \\
      & =
      \begin{pmatrix}
        \I \Id_d & \Id_d
      \end{pmatrix}
      \bigl((F^{\kappa_t})^{\TT} F^{\kappa_t}\bigr){(q,p)}
      \begin{pmatrix}
        - \I \Id_d \\
        \Id_d
      \end{pmatrix} \\
      & \quad +
      \begin{pmatrix}
        \I \Id_d & \Id_d
      \end{pmatrix}
      (F^{\kappa_t})^{\TT}{(q,p)}
      \begin{pmatrix}
        0 & -\I \Id_d \\
        \I \Id_d & 0
      \end{pmatrix}
      F^{\kappa_t}{(q,p)}
      \begin{pmatrix}
        -\I \Id_d \\
        \Id_d
      \end{pmatrix}  \\
      & =
      \begin{pmatrix}
        \I \Id_d & \Id_d
      \end{pmatrix}
      \bigl((F^{\kappa_t})^{\TT} F^{\kappa_t}\bigr){(q,p)}
      \begin{pmatrix}
        -\I \Id_d \\
        \Id_d
      \end{pmatrix} + 2\Id_d.
    \end{aligned}
  \end{equation*}
  In the last equality, we have used the property of symplecticity of
  $F^{\kappa_t}{(q,p)}$.  Therefore $\bigl(Z^{\kappa_t}
  (Z^{\kappa_t})^{\ast}\bigr){(q,p)}$ is positive definite, which
  implies $\det\bigl(Z^{\kappa_t}{(q,p)}\bigr)$ is bounded away
  uniformly from zero for $\abs{p} > 0$. The invertibility of
  $Z^{\kappa_t}{(q,p)}$ and bounds \eqref{eq:invZbound} follows from
  the representation of $(Z^{\kappa_t})^{-1}{(q,p)}$ by minors and
  Proposition~\ref{prop:Jacobian}.
\end{proof}

The following Lemma plays an important role in frozen Gaussian
approximation.

\begin{lemma}\label{lem:veps1}
  For each $n = 1, \cdots, N$, $ t\in [0, T]$, let $ b(y, q, p):
  \RR^{3d} \to {\CC^{N}}$ such that for any $y \in \RR^d$, $\supp b(y,
  \cdot, \cdot) \subset {K_{\delta}}$. Assume that there exists some
  $m \in \RR$ that for any $k \in \NN$,
  \begin{equation*}
    \sup_{y \in \RR^d} (1 + \abs{y}^2)^{-m/2}
    \Lambda_{k, \delta}[b(y, \cdot, \cdot)]
    < \infty.
  \end{equation*}
  Then for a multi-index $(j_1, \cdots, j_{\nu})$ with $\nu \geq 1$,
  \begin{equation}\label{eq:con3}
    ({x} - {Q}_n(t, q, p))_{j_1}\cdots ({x} - {Q}_n(t, q, p))_{j_{\nu}}
    b{(y,q,p)} \sim \mc{T}^{\nu,\veps}_{n, j_1\cdots j_{\nu}}b{(y,q,p)}.
  \end{equation}
  Here $\mc{T}^{\nu,\veps}_{n, j_1\cdots j_{\nu}}$ is defined by the
  recursive relation \eqref{op:T1}-\eqref{op:Tnu}, corresponding to
  the $n$-th branch. Here, we have used the notation $ f \sim g $, if
  \begin{equation}
    \int_{\RR^{3d}} f(y,q,p) e^{\frac{\I}{\veps} \Phi_n{(t,x,y,q,p)}} \ud y \ud p \ud q
    = \int_{\RR^{3d}} g(y,q,p) e^{\frac{\I}{\veps} \Phi_{n}{(t,x,y,q,p)}} \ud y \ud p \ud
    q.
  \end{equation}
\end{lemma}

\begin{proof}

  We omit the subscript $n$ in the proof for simplicity, because the
  argument is the same for each branch.

  Observe that at $t=0$, for $\abs{p} > 0$,
  \begin{align*}
    & \partial_{q} S(0, q, p) - \partial_q Q(0, q, p) P(0, q, p) + p=0, \\
    & \partial_{p} S(0, q, p) - \partial_p Q(0, q, p) P(0, q, p) =0.
  \end{align*}
  Using \eqref{eq:Hflow} and \eqref{eq:S}, we have
  \begin{align*}
    \frac{\ud}{\ud t}(\partial_{q}S-\partial_{q} Q P + p) & = \partial_{q}
    (\partial_t S)-\partial_{q} (\partial_t Q) P-
    \partial_{q} Q \partial_t P \\
    & = \partial_q (P \cdot \partial_P H - H) -\partial_q (\partial_P
    H) P + \partial_q Q \partial_Q H \\
    & = \partial_q P \partial_P H- ( \partial_q Q
    \partial_Q + \partial_q P \partial_P) H +\partial_q Q \partial_Q H
    = 0,
  \end{align*}
  {where $S$, $Q$ and $P$ are evaluated at $(t,q,p)$, and
  $\partial_QH$, $\partial_PH$ are evaluated at $(Q,P)$.}

  Analogously we have
  \[
  \frac{\ud}{\ud t}\bigl(\partial_p S{(t,q,p)}
   - \partial_p Q{(t,q,p)} P{(t,q,p)}\bigr) = 0.
  \]
  Therefore for all $t \in [0, T]$, we have
  \begin{align}\label{eq:Sq}
    & \partial_q S{(t,q,p)} - \partial_q Q{(t,q,p)} P{(t,q,p)} + p = 0,
    \\ \label{eq:Sp}
    & \partial_p S{(t,q,p)} - \partial_p Q{(t,q,p)} P{(t,q,p)} = 0.
  \end{align}

  Straightforward calculations yield
  \begin{align*}
    & \partial_q \Phi{(t,x,y,q,p)} = \bigl(\partial_q P{(t,q,p)}
 - \I \partial_q Q{(t,q,p)}\bigr)\bigl(x-Q{(t,q,p)}\bigr) -\I
    (y-q), \\
    & \partial_p \Phi{(t,x,y,q,p)} = \bigl(\partial_p P{(t,q,p)}
  - \I \partial_p Q{(t,q,p)}\bigr) \bigl(x-Q{(t,q,p)}\bigr)- (y-q).
  \end{align*}
  This implies that
  \begin{equation}\label{eq:dzPhi}
    \I\partial_z \Phi{(t,x,y,q,p)}=Z{(t,q,p)}\bigl(x-Q{(t,q,p)}\bigr),
  \end{equation}
  where $\partial_z$ and $Z{(t,q,p)}$ are defined in \eqref{eq:op_zZ}.

  Using \eqref{eq:dzPhi} and Lemma~\ref{lem:Zinv} for the
  invertibility of $Z{(t,q,p)}$ on the support of $b{(y,q,p)}$,
  one has, for $j = 1, \cdots, d$,
  \begin{multline*}
    \int_{\RR^{3d}} \bigl(x - Q{(t,q,p)}\bigr)_j e^{\frac{\I}{\veps} \Phi{(t,x,y,q,p)}} b(y, q, p) \ud
    y\ud q\ud p \\ = \veps\int_{\RR^{3d}} Z^{-1}_{jk}{(t,q,p)}
    \left(\frac{\I}{\veps}\partial_{z_k}\Phi{(t,x,y,q,p)}\right)
    e^{\frac{\I}{\veps} \Phi{(t,x,y,q,p)}}  b(y, q, p) \ud y\ud q \ud p.
  \end{multline*}
  We remark that the integrability of the above integral follows from
  the exponential decay in $y$ of $\exp(\I \Phi{(t,x,y,q,p)}/\veps)$ and compact
  support of $b{(y,q,p)}$ in $(q, p)$. Integration by parts gives
  \begin{align*}
    \int_{\RR^{3d}} \bigl(x &- Q{(t,q,p)}\bigr)_j e^{\frac{\I}{\veps} \Phi{(t,x,y,q,p)}} b(y, q, p) \ud
    y\ud q\ud p \\ &= -\veps \int_{\RR^{3d}} \partial_{z_k}\bigl( Z^{-1}_{jk}{(t,q,p)}
    b{(y,q,p)}\bigr) e^{\frac{\I}{\veps} \Phi{(t,x,y,q,p)}} \ud y \ud q \ud p \\
    & = \int_{\RR^{3d}} \veps\mc{D}_j b{(y,q,p)} e^{\frac{\I}{\veps} \Phi{(t,x,y,q,p)}} \ud y \ud q \ud p
    \\
    & = \int_{\RR^{3d}} \mc{T}^{1,\veps}_j b{(y,q,p)} e^{\frac{\I}{\veps} \Phi{(t,x,y,q,p)}} \ud y \ud q \ud p.
  \end{align*}
  This proves \eqref{eq:con3} for $\nu=1$.

  Making use of the above equality twice produces \eqref{eq:con3} for $\nu=2$,
  \begin{align*}
    \bigl(x - Q{(t,q,p)}\bigr)_{j_1} \bigl(x &- Q{(t,q,p)}\bigr)_{j_2} b{(y,q,p)}
    \\ & \sim \veps\mc{D}_{j_2} \Bigl(\bigl(x-Q{(t,q,p)}\bigr)_{j_1} b{(y,q,p)}\Bigr) \\
    & = \veps \mc{G}_{j_1j_2} b{(y,q,p)} +\veps
    \bigl(x-Q{(t,q,p)}\bigr)_{j_1} \mc{D}_{j_2} b{(y,q,p)}
    \\
    & \sim\veps \mc{G}_{j_1j_2} b{(y,q,p)} +\veps^2
    \mc{D}_{j_1} \mc{D}_{j_2} b{(y,q,p)}
    \\& =\mc{T}^{2,\veps}_{j_1j_2}b{(y,q,p)}.
  \end{align*}
  Furthermore, we have
  \begin{equation*}
    \begin{aligned}
      \bigl({x} &- {Q}{(t,q,p)}\bigr)_{j_1}\cdots \bigl({x} - {Q}{(t,q,p)}\bigr)_{j_{\nu}} b{(y,q,p)}
     \\ & \sim\veps
      \mc{D}_{j_{\nu}}\Bigl(\bigl({x} - {Q}{(t,q,p)}\bigr)_{j_1}\cdots
      \bigl({x} - {Q}{(t,q,p)}\bigr)_{j_{\nu-1}} b{(y,q,p)}\Bigr) \\
      & =\veps \bigl({x} - {Q}{(t,q,p)}\bigr)_{j_1}\cdots \bigl({x} - {Q}{(t,q,p)}\bigr)_{j_{\nu-1}}
      \mc{D}_{j_{\nu}}b{(y,q,p)} \\
      & \qquad +\veps\sum_{\alpha=1}^{\nu-1}\bigl({x} - {Q}{(t,q,p)}\bigr)_{j_1}\cdots
\bigl({x} - {Q}{(t,q,p)}\bigr)_{j_{\alpha-1}} \\
      & \qquad \qquad \times \bigl({x} - {Q}{(t,q,p)}\bigr)_{j_{\alpha+1}}\cdots
     \bigl({x} - {Q}{(t,q,p)}\bigr)_{j_{\nu-1}} \mc{G}_{j_{\alpha}j_{\nu}}b{(y,q,p)} \\
      & \sim\veps\mc{T}^{\nu-1,\veps}_{j_1j_2\cdots
        j_{\nu-1}}\mc{D}_{j_{\nu}}b{(y,q,p)} \\
      & \qquad
      +\veps\sum_{\alpha=1}^{\nu-1}\mc{T}^{\nu-2,\veps}_{j_1\cdots
        j_{\alpha-1}j_{\alpha+1}\cdots
        j_{\nu-1}}\mc{G}_{j_{\alpha}j_{\nu}}b{(y,q,p)} \\
      & =\mc{T}^{\nu,\veps}_{j_1\cdots j_{\nu}}b{(y,q,p)},
    \end{aligned}
  \end{equation*}
  where the last step follows by the recursive relation
  \eqref{op:Tnu}.  This proves \eqref{eq:con3} for $\nu\geq 3$.

\end{proof}

To construct a high order approximation to the solution, we need a
filtered version of frozen Gaussian approximation, defined as
follows. The relationship between $\wt{\mc{P}}^{\veps}_{t, K, \delta}$
and $\mc{P}^{\veps}_{t, K, \delta}$ will be addressed in
Lemma~\ref{lem:ptilde2} in the end of this section.
\begin{equation}\label{eq:ansatzdelta}
  \begin{aligned}
    (\wt{\mc{P}}_{t,K, \delta}^{\veps} u_0)(x) & = \sum_{n=1}^N
    \left(\mc{I}_n^{\veps}\bigl(t, \wt{\mc{M}}^{\veps}_{n, K, t, \delta}
      \bigr) u_0{(y)}
    \right)(x) \\
    & = \frac{1}{(2\pi \veps)^{3d/2}} \sum_{n=1}^N \int_{\RR^{3d}}
    e^{\I\Phi_n{(t,x,y,q,p)}/\veps} \wt{\mc{M}}^{\veps}_{n,K, t, \delta}(q, p)
    u_{0}(y) \ud q \ud p \ud y.
  \end{aligned}
\end{equation}
Here the matrix-valued function $\wt{\mc{M}}^{\veps}_{n,K, t, \delta}:
\RR^{2d} \to \CC^{N \times N}$ is given by
\begin{equation}
  \wt{\mc{M}}^{\veps}_{n,K,t,\delta}(q, p) =
  \wt{\mc{M}}^{\veps}_{n,K,\delta}(t, q, p) = \sum_{k=0}^{K-1}\veps^k\wt{M}_{n,k,\delta}(t,q,p).
\end{equation}
For $(q, p) \in \RR^{2d}$ and $t \in [0, T]$,
$\wt{M}_{n,k,\delta}(t,q,p)$ is defined as follows. For $k = 0$,
\begin{equation}\label{eq:wtM0}
  \wt{M}_{n, 0,\delta}(t,q,p)=
  \wt{\sigma}_{n, 0,\delta}(t,q,p)
  R_n\bigl(Q_n{(t,q,p)},P_n{(t,q,p)}\bigr)L_n^{\TT}(q,p)
  \chi_{\delta}(q, p),
\end{equation}
where $\chi_{\delta}(q, p)$ is the filter function given before, and
$\wt{\sigma}_{n, 0,\delta}(t,q,p)$ satisfies
\begin{equation}\label{eq:wtsigma_0}
  \frac{\ud}{\ud t}\wt{\sigma}_{n, 0,\delta}(t,q,p)
  \chi_{\delta}(q, p)+\wt{\sigma}_{n, 0,\delta}(t,q,p)
  \lambda_n(t,q,p)\chi_{\delta}(q, p) = 0,
\end{equation}
with the initial condition $\wt{\sigma}_{n, 0, \delta}(0, q, p) = 2^{d/2}$
and $\lambda_n(t,q,p)$ is given by \eqref{eq:lambda}.

For $k \geq 1$, $\wt{M}_{n,k,\delta}$ is given recursively by
\begin{equation}\label{eq:wtMk}
  \wt{M}_{n, k,\delta}(t,q,p)=
  \wt{\sigma}_{n, k,\delta}(t,q,p)R_n(Q_n,P_n)L_n(q,p)^\TT\chi_{\delta}(q, p)
  +\wt{M}_{n,k,\delta}^\perp(t,q,p),
\end{equation}
where $\wt{M}_{n,k,\delta}^\perp(t,q,p)$ is given by
\begin{equation}\label{eq:wtMkperp}
  \wt{M}_{n, k,\delta}^\perp{(t,q,p)}=\mc{L}_{n,0}^{\dagger}(Q_n,P_n)
  \sum_{s=1}^k \bigl(\mc{L}_{n,s} \wt{M}_{n, k-s,\delta}\bigr){(t,q,p)},
\end{equation}
and $\wt{\sigma}_{n, k,\delta}{(t,q,p)}$ satisfies
\begin{equation}\label{eq:wtsigma_k}
\begin{aligned}
  \frac{\ud}{\ud t} \wt{\sigma}_{n, k,\delta}{(t,q,p)}\chi_{\delta}(q, p)&+
  \wt{\sigma}_{n, k,\delta}{(t,q,p)} \lambda_n(t,q,p)\chi_{\delta}(q, p) \\&
  + L_n(Q_n,P_n)^\TT\biggl(
  (\mc{L}_{n,1}\wt{M}_{n,k,\delta}^\perp){(t,q,p)}
\\& \qquad +\sum_{s=2}^{k+1} (\mc{L}_{n,
    s}\wt{M}_{n, k+1-s,\delta}){(t,q,p)} \biggr)R_n(q,p)=0,
\end{aligned}
\end{equation}
with the initial condition $\wt{\sigma}_{n,k,\delta}(0, q, p) = 0$.  In
\eqref{eq:wtMkperp} and \eqref{eq:wtsigma_k}, $Q_n$ and $P_n$ are
evaluated at $(t,q,p)$. We remark that, from the definitions
\eqref{eq:wtM0} and \eqref{eq:wtMk}, it is clear that the value of
$\wt{\sigma}_{n, k, \delta}{(t,q,p)}$ outside $\supp \chi_{\delta}(q,
p) \subset K_{\delta/2}$ will not affect the value of $\wt{M}_{n, k,
  \delta}{(t,q,p)}$.

\begin{lemma}\label{lem:L0}
  For each $n = 1, \cdots, N$, $k \in \NN$ and any $\delta > 0$, there
  exists constant $C_{k, \delta}$ that
  \begin{equation}\label{eq:boundL0}
    \Lambda_{k, \delta}[\mc{L}_{n, 0}^{\dagger}{(q,p)}] \leq C_{k, \delta}.
  \end{equation}
\end{lemma}

\begin{proof}

  Recall that $\mc{L}_{n, 0}^{\dagger}{(q,p)}$ is defined in
  \eqref{eq:L0dagger} as
  \begin{equation*}
    \mc{L}_{n,0}^{\dagger}(q, p) = \I \sum_{m\neq n} \bigl(H_n(q, p)
    - H_m(q, p)\bigr)^{-1}
    R_m(q, p) L_m^{\TT}(q, p).
  \end{equation*}
  By strict hyperbolicity and compactness, there exists $g_{\delta} >
  0$ such that
  \begin{equation}
    \inf_{(q, p) \in K_{\delta}} \inf_{n \neq m}
    \abs{H_n(q, p) - H_m(q, p)} \geq g_{\delta}.
  \end{equation}
  The estimate \eqref{eq:boundL0} follows easily.

\end{proof}

\begin{lemma}\label{lem:wtMbound}
  For each $n = 1, \cdots, N$, any $t \in [0, T]$ and any $k \in \NN$,
  we have $\supp \wt{M}_{n, k, \delta}{(t,\cdot,\cdot)} \subset
  K_{\delta/2}$, $\supp \partial_t \wt{M}_{n, k,
    \delta}{(t,\cdot,\cdot)} \subset K_{\delta/2}$. Moreover, for any
  $s \in \NN$, there exists constant $C_{k,s, \delta, T}$ that
  \begin{align*}
    & \sup_{t \in [0, T]} \Lambda_{s, \delta/2}[ \wt{M}_{n, k,
      \delta}{(t,q,p)} ]
    \leq C_{k, s, \delta, T}; \\
    & \sup_{t \in [0, T]} \Lambda_{s, \delta/2}[ \partial_t \wt{M}_{n,
      k, \delta}{(t,q,p)}] \leq C_{k, s, \delta, T}.
  \end{align*}
\end{lemma}

\begin{proof}

  The argument is the same for each branch, so we will omit the
  subscript $n$ in the proof.

  By Proposition~\ref{prop:boundaway}, there exists constant
  $\delta_T$ such that for $t \in [0, T]$, we have $\bigl(Q(t,q,p),
  P(t, q, p)\bigr) \in K_{\delta_T}$ for $(q, p) \in
  K_{\delta/2}$. From Lemma~\ref{lem:Zinv} and
  Proposition~\ref{prop:Hbound}, it is easy to conclude that for any
  $s \in \NN$ and $l = 0, 1$,
  \begin{equation}\label{eq:lambdabound}
    \sup_{t\in [0, T]}
    \Lambda_{s, \delta/2}[\partial_t^l \lambda{(t,q,p)}] \leq C_{s, \delta, T}.
  \end{equation}

  Equation \eqref{eq:wtM0} implies that
  \begin{equation*}
    \supp \wt{M}_{0, \delta}{(t,q,p)}
    \subset \supp \chi_{\delta}{(q,p)} \subset K_{\delta/2}.
  \end{equation*}
  Moreover, \eqref{eq:wtsigma_0} implies that
  \begin{equation*}
    \frac{\ud}{\ud t} \abs{\wt{\sigma}_{0, \delta}{(t,q,p)} \chi_{\delta}{(q,p)}}
    \leq \abs{\lambda(t, q, p)} \abs{\wt{\sigma}_{0, \delta}{(t,q,p)}
      \chi_{\delta}{(q,p)}}.
  \end{equation*}
  Hence Gronwall's inequality yields
  \begin{equation*}
    \sup_{t\in [0, T]} \sup_{(q,p) \in K_{\delta/2}}
    \abs{\wt{\sigma}_{0, \delta}(t, q, p) \chi_{\delta}(q, p)}
    \leq C_{\delta, T}.
  \end{equation*}
  Taking derivatives of \eqref{eq:wtsigma_0} with respect to $(t, q,
  p)$, we have for multi-index $\alpha$ and $l = 0, 1$,
  \begin{multline*}
    \frac{\ud}{\ud t} \partial^{\alpha}_{(q, p)} \partial_t^l
    \bigl(\wt{\sigma}_{0, \delta}{(t,q,p)} \chi_{\delta}{(q,p)}\bigr) \\
    = \sum_{\beta \leq \alpha} \sum_{0\leq m \leq l}
    {\alpha\choose\beta} \partial_{(q, p)}^{\beta} \partial_t^{m}
    \lambda(t, q, p)
    \partial_{(q, p)}^{\alpha - \beta} \partial_t^{l-m}
    \bigl(\wt{\sigma}_{0, \delta}{(t,q,p)} \chi_{\delta}{(q,p)}\bigr).
  \end{multline*}
  Using \eqref{eq:lambdabound} and by induction, we have for $l = 0,
  1$,
  \begin{equation*}
    \sup_{t \in [0, T]} \Lambda_{s, \delta/2}\bigl[\partial_t^{l}
    \bigl(\wt{\sigma}_{0, \delta}{(t,q,p)} \chi_{\delta}{(q,p)}\bigr)\bigr]
    \leq C_{s, \delta, T},
  \end{equation*}
  which implies
  \begin{equation}\label{eq:wtM0bound}
    \sup_{t \in [0, T]} \Lambda_{s, \delta/2}
    [\partial_t^l\wt{M}_{0, \delta}{(t,q,p)}]
    \leq C_{s, \delta, T}.
  \end{equation}

  Next, we consider $\wt{M}_{1, \delta}{(t,q,p)}$. \eqref{eq:wtMkperp} gives
  \begin{equation*}
    \wt{M}_{1, \delta}^{\perp}{(t,q,p)} = \mc{L}_{0}^{\dagger}
    {\bigl(Q_n(t,q,p),P_n(t,q,p)\bigr)}
    (\mc{L}_{1}\wt{M}_{0, \delta}){(t,q,p)}.
  \end{equation*}
  As $\supp \wt{M}_{0, \delta}{(t,q,p)} \subset K_{\delta/2}$, we have $\supp
  \wt{M}_{1, \delta}^{\perp}{(t,q,p)} \subset K_{\delta/2}$.  By
  Lemma~\ref{lem:L0}, \eqref{op:L1} and \eqref{eq:wtM0bound}, it is
  not hard to see that for $l = 0, 1$,
  \begin{equation*}
    \sup_{t \in [0, T]} \Lambda_{s, \delta/2}
    [\partial_t^l\wt{M}^{\perp}_{1, \delta}{(t,q,p)}]
    \leq C_{s, \delta, T}.
  \end{equation*}
  Similarly, we get from \eqref{eq:wtsigma_k} the estimate
  \begin{equation*}
    \sup_{t \in [0, T]} \Lambda_{s, \delta/2}[\partial_t^{l}
    \wt{\sigma}_{1, \delta}{(t,q,p)} \chi_{\delta}(q,p)]
    \leq C_{s, \delta, T}.
  \end{equation*}
  In summary, we obtain $\supp \wt{M}_{1, \delta}{(t,q,p)} \subset
  K_{\delta/2}$, and for $l = 0, 1$,
  \begin{equation*}
    \sup_{t \in [0, T]} \Lambda_{s, \delta/2}
    [\partial_t^l\wt{M}_{1, \delta}{(t,q,p)}]
    \leq C_{s, \delta, T}.
  \end{equation*}

  Continuing this procedure, we obtain the estimates for higher order
  asymptotic terms: $\supp \wt{M}_{k, \delta}{(t,q,p)} \subset \supp
  K_{\delta/2}$ for $k \geq 0$, and
  \begin{align*}
    & \sup_{t \in [0, T]} \Lambda_{s, \delta/2} [\wt{M}_{k, \delta}{(t,q,p)}]
    \leq C_{k, s, \delta, T}, \\
    & \sup_{t \in [0, T]} \Lambda_{s, \delta/2} [\partial_t \wt{M}_{k,
      \delta}{(t,q,p)}] \leq C_{k, s, \delta, T}.
  \end{align*}
  This proves the Lemma.

\end{proof}

We now show that the filtered frozen Gaussian approximation gives a
high order approximation to the solution.
\begin{prop}\label{prop:highorder}
  Under the same assumption of Theorem~\ref{thm:main}, for any $T>0$,
  $K\in \NN$, there exists constant $C_{T,K}$, so that for any
  $\veps>0$,
  \begin{equation*}
    \norm{\bigl(\partial_t+ A_l(x)\partial_{x_l}\bigr)
      \wt{\mc{P}}_{t, K, \delta}^{\veps} u_0^{\veps}}_{L^2(\RR^d)}
    \leq  C_{T, K} M \veps^{K-1}.
  \end{equation*}
\end{prop}

\begin{proof}

  Without further indication, $S_n$, $Q_n$ and $P_n$ are
  all evaluated at $(t,q,p)$ in the proof.

  For each $l = 1, \cdots, d$, Taylor expansion of $A_l(x)$ around
  $x=Q_n$ gives
  \begin{multline}\label{eq:Ataylor}
    A_l(x) = A_l(Q_n) + \sum_{\nu=1}^{2K} \frac{1}{\nu
      !}(x-Q_n)_{j_1}\cdots (x-Q_n)_{j_{\nu}}
    \partial^{\nu}_{{j_1}\cdots
      {j_{\nu}}}A_l(Q_n)\\+ \frac{1}{(2K)!}(x-Q_n)_{j_1}\cdots
    (x-Q_n)_{j_{2K+1}}\widehat{A}_{l,j_1\cdots j_{2K+1}}(Q_n),
  \end{multline}
  where $\widehat{A}_{l,j_1\cdots j_{2K+1}}(Q_n)$ is the integral
  \begin{equation}
    \widehat{A}_{l,j_1\cdots j_{2K+1}}(Q_n)=\int_0^1(1-\tau)^{2K}
    \partial^{2K+1}_{j_1\cdots j_{2K+1}}A_l\bigl(Q_n+\tau(x-Q_n)\bigr)
    \ud \tau.
  \end{equation}
  We define the operators $\wt{\mc{L}}_{n,k}$ on $C^1([0, T],
  C^{\infty}(\RR^{2d}; \CC^{N\times N}))$ as follows. For $M:
  \RR\times \RR^{2d} \to \CC^{N\times N}$, we define for $0\leq k\leq
  K$,
  \begin{align*}
    (\wt{\mc{L}}_{n,k}M){(t,q,p)}&=(\mc{L}_{n,k}M){(t,q,p)}, \\
    \intertext{for $K+1\leq k\leq 2K+2$,}
    (\wt{\mc{L}}_{n,k}M){(t,q,p)} & =
    \sum_{\nu=k}^{2K}\mc{T}_{n,j_1\cdots j_{\nu}}^{\nu,k}
    \biggl(\Bigl( -\frac{1}{(\nu-1)!}\partial^{\nu-1}_{{j_1}\cdots
      {j_{\nu-1}}}A_{j_\nu}(Q_n)\\
    & \hspace{8em}+\frac{\I}{\nu!}
    P_l\partial^{\nu}_{{j_1}\cdots {j_{\nu}}}A_l(Q_n)\Bigr) M{(t,q,p)}\biggr) \\
    & +\mc{T}_{n,j_1\cdots j_{2K+1}}^{2K+1,k} \biggl(\Bigl(
    -\frac{1}{(2K)!}\partial^{2K}_{{j_1}\cdots
      {j_{2K}}}A_{j_{2K+1}}(Q_n)\\
    & \hspace{8em} +\frac{\I}{(2K)!}  P_l\widehat{A}_{l,j_1\cdots
      j_{2K+1}}(Q_n)\Bigr)
    M{(t,q,p)}\biggr) \\
    & +\mc{T}_{n,j_1\cdots j_{2K+2}}^{2K+2,k} \biggl( -\frac{1}{(2K)!}
    \widehat{A}_{j_{2K+2},j_1\cdots j_{2K+1}}(Q_n) M{(t,q,p)}\biggr), \\
    \intertext{and for $k\geq 2K+3$,}
    (\wt{\mc{L}}_{n,k}M){(t,q,p)} & = 0.
  \end{align*}
  Differentiating $\Phi_n{(t,x,y,q,p)}$ with respect to $t$ and $x$ gives
  \begin{align*}
    & \partial_t \Phi_n{(t,x,y,q,p)} = \partial_t S_n - P_n \cdot \partial_t Q_n
    + (x - Q_n) \cdot (\partial_t P_n - \I \partial_t Q_n), \\
    & \partial_{x} \Phi_n{(t,x,y,q,p)} = \I (x - Q_n) + P_n.
  \end{align*}

  Combining these expressions, substituting \eqref{eq:ansatzdelta} and
  \eqref{eq:Ataylor} into \eqref{eq:hypersys}, and organizing terms in
  orders of $\veps$ produce, after straightforward calculations,
  \begin{equation*}
    (\partial_t + A_l(x) \partial_{x_l}) (\wt{\mc{P}}_{t, K, \delta}^{\veps}
    u_0^{\veps}){(x)}
    = \sum_{k=0}^{K-1} \veps^{k-1} v_{t, k, \delta}{(x)} + \veps^{K-1}
    r_{t, K, \delta}{(x)},
  \end{equation*}
  where we have, 
  \begin{align*}
    & v_{t, 0, \delta}{(x)} = \sum_{n=1}^N (2\pi \veps)^{-3d/2}
    \int_{\RR^{3d}} (\wt{\mc{L}}_{n,0} \wt{M}_{n, 0,\delta} ){(t,q,p)}
    e^{\I
      \Phi_n/\veps} u_0^\veps(y) \ud y \ud p \ud q, \\
    & v_{t, 1, \delta}{(x)} = \sum_{n=1}^N (2\pi \veps)^{-3d/2}
    \int_{\RR^{3d}} \bigl((\wt{\mc{L}}_{n, 0}\wt{M}_{n, 1,\delta}){(t,q,p)} \\
    & \nn \phantom{v_{t, 1, \delta} = \sum_{n=1}^N (2\pi
      \veps)^{-3d/2} \int_{\RR^{3d}} } +(\wt{\mc{L}}_{n, 1}\wt{M}_{n,
      0,\delta}){(t,q,p)}
    \bigr) e^{i\Phi_n/\veps} u_0^\veps(y)\ud y\ud p\ud q, \\
    & v_{t, k, \delta}{(x)} = \sum_{n=1}^N (2\pi \veps)^{-3d/2}
    \int_{\RR^{3d}} \sum_{s=0}^k (\wt{\mc{L}}_{n, s}\wt{M}_{n,
      k - s,\delta}){(t,q,p)} e^{i\Phi_n/\veps} u_0^\veps(y)\ud y\ud p\ud q, \\
    \intertext{and the remainder,} & r_{t,K,\delta}{(x)} =
    \sum_{k=0}^{2K+1} \veps^{k}
    \sum_{s=0}^{K-1} \sum_{n=1}^N (2\pi \veps)^{-3d/2} \\
    & \nn \hspace{6em} \times \int_{\RR^{3d}}
    (\wt{\mc{L}}_{n,k+s+1}\wt{M}_{n,K-s-1,\delta}){(t,q,p)}e^{\I\Phi_n/\veps}
    u_{0}^\veps(y) \ud y \ud p \ud q.
  \end{align*}
  Here $\Phi_n$ is evaluated at $(t, x, y, q, p)$.

  We will show that $v_{t, k, \delta} = 0$ for $k = 0, \cdots, K-1$ and
  control the norm of the remainder $r_{t, K, \delta}$ to prove the
  proposition.

  By definitions of $R_n(Q_n, P_n)$ as in \eqref{eigen:right} and of
  $S_n$ as in \eqref{eq:S}, we get
  \begin{multline*}
    \Bigl( \partial_t S_n
    - P_n\cdot \partial_t Q_n + P_{n,l} A_l(Q_n) \Bigr) R_n(Q_n, P_n) \\
    = \Bigl(\partial_t S_n - P_n\cdot \partial_t Q_n + H_n(Q_n,
    P_n)\Bigr) R_n(Q_n, P_n) = 0,
  \end{multline*}
  which implies $(\wt{\mc{L}}_{n, 0} \wt{M}_{n,0,\delta}){(t,q,p)} =
  0$ for each $n$, since $\wt{M}_{n, 0, \delta}$ is given as
  \eqref{eq:wtM0}.  Therefore, $v_{t, 0, \delta} = 0$.

  To prove $v_{t, 1, \delta} = 0$, it suffices to show that, for each
  $m=1,\cdots,N$,
  \begin{equation}\label{eq:v1}
    L_m^\TT(Q_n,P_n)\biggl((\wt{\mc{L}}_{n, 0}\wt{M}_{n, 1,\delta}){(t,q,p)}+
    (\wt{\mc{L}}_{n, 1}\wt{M}_{n, 0,\delta}){(t,q,p)}\biggr)=0.
  \end{equation}

  When $m\neq n$, \eqref{eq:v1} is equivalent to, by \eqref{eigen:right},
  \begin{equation}
    L_m^\TT(Q_n,P_n)\biggl( (\wt{\mc{L}}_{n, 0}
    \wt{M}_{n, 1,\delta}^{\perp}){(t,q,p)}+
    (\wt{\mc{L}}_{n, 1}\wt{M}_{n, 0,\delta}){(t,q,p)}\biggr)=0,
  \end{equation}
  which is valid by \eqref{eq:L0dagger} and \eqref{eq:wtMkperp}.

  When $m=n$,  \eqref{eq:v1} is equivalent to, by \eqref{eigen:left},
  \begin{equation}\label{eq:L1}
    L_n^\TT(Q_n,P_n) (\wt{\mc{L}}_{n, 1}\wt{M}_{n, 0,\delta}){(t,q,p)}=0,
  \end{equation}
  which will be proved as follows. For convenience, we will omit the
  index $n$ in the following calculations. Substituting
  \eqref{eq:wtM0} in \eqref{eq:L1} and using \eqref{op:T_j},
  \eqref{op:T_j1j2} and \eqref{op:L1}, we can rewrite \eqref{eq:L1} as
  \begin{equation}\label{eq:1st}
    \begin{aligned}
      {L}^{\TT}({Q}, {P}) \biggl( & \partial_t
      \bigl(\wt{\sigma}_{0,\delta}(t,q,p)
      R(Q,P)\bigr) L^{\TT}({q}, {p})\chi_{\delta}(q, p) \\
      & - \partial_{z_k} \Bigl( \bigl( \I \partial_t P_j + \partial_t
      Q_j - A_j(Q) + \I P_l {\partial_{j}} A_l(Q) \bigr)\\ &
      \hspace{4em}\times \wt{\sigma}_{0,\delta}(t,q,p) R(Q,P)
      Z_{jk}^{-1} L^{\TT}({q}, {p})\chi_{\delta}(q, p)
      \Bigr) \\
      & + \partial_{z_s} Q_j \bigl( - {\partial_{j}} A_k(Q) +
      \frac{\I}{2} P_l {\partial^2_{jk}} A_l(Q) \bigr) \\ &
      \hspace{4em} \times\wt{\sigma}_{0,\delta}(t,q,p) R(Q,P)
      Z_{ks}^{-1} L^{\TT}({q}, {p})\chi_{\delta}(q, p) \biggr) =
      0.
    \end{aligned}
  \end{equation}
  For the first term, easy calculations yield
  \begin{equation*}
  \begin{aligned}
    {L}^{\TT}({Q}, {P}) &\Bigl(\partial_t
    \bigl(\wt{\sigma}_{0,\delta}(t,q,p) R(Q,P)\bigr) L^{\TT}({q},
    {p})\chi_{\delta}(q, p)\Bigr) \\ & = \biggl(\partial_t
    \wt{\sigma}_{0,\delta}{(t,q,p)} +
    \wt{\sigma}_{0,\delta}{(t,q,p)} {L}^{\TT}(Q,P) \Bigl(\partial_{{P}} H(Q,P)
    \cdot \partial_{{Q}} {R}(Q,P) \\&\hspace{6em} - \partial_{{Q}} H(Q,P) \cdot
    \partial_{{P}} {R}(Q,P)\Bigr)\biggr)L^{\TT}(q, p)\chi_{\delta}(q,
    p).
  \end{aligned}
  \end{equation*}
  To simplify the second term in \eqref{eq:1st}, we observe that
  differentiating \eqref{eigen:right} with respect to $P$ and $Q$
  yields,
  \begin{equation}
    \begin{aligned}
      & A_l({Q}) {R}({Q}, {P}) + P_j
      A_j({Q}) \partial_{P_l} {R}({Q}, {P}) \\
      & \hspace{6em} = \partial_{P_l} H({Q}, {P}) {R}({Q}, {P}) +
      H({Q}, {P}) \partial_{P_l} {R}({Q}, {P}), \nn
    \end{aligned}
  \end{equation}
  and
  \begin{equation}
    \begin{aligned}
      & P_j \partial_{l} A_j({Q}) {R}({Q}, {P}) + P_j
      A_j({Q}) \partial_{Q_l}
      {R}({Q}, {P}) \\
      & \hspace{6em} = \partial_{Q_l} H({Q}, {P}) {R}({Q}, {P}) +
      H({Q}, {P}) \partial_{Q_l} {R}({Q}, {P}). \nn
    \end{aligned}
  \end{equation}
  Taking inner product with ${L}({Q}, {P})$ on the left produces
  \begin{align}
    \label{eq:derivP} & {L}^{\TT}({Q}, {P}) \bigl( A_l({Q})
    - \partial_{P_l} H({Q}, {P}) \bigr) {R}({Q}, {P}) = 0, \\
    \label{eq:derivQ} & {L}^{\TT}({Q}, {P}) \bigl( P_j \partial_{l}
    A_j({Q}) - \partial_{Q_l} H({Q}, {P}) \bigr) {R}({Q}, {P}) = 0.
  \end{align}
  Define the short hand notation
  \begin{equation*}
    \begin{aligned}
      {F}_j{(t,q,p)} &= (\I \partial_t P_j + \partial_t Q_j) {R}(Q,P)
      - A_j(Q) {R}(Q,P) + \I P_l \partial_{j} A_l(Q) {R}(Q,P) \\
      & = - \Bigl( \bigl(A_j(Q) - \partial_{P_j} H(Q,P)\bigr)
   + \I \bigl( \partial_{Q_j} H(Q,P) -
      P_l \partial_{j} A_l(Q) \bigr) \Bigr) {R}(Q,P).
    \end{aligned}
  \end{equation*}
  Using \eqref{eq:derivP} and \eqref{eq:derivQ}, it is clear that for
  any $j = 1, \ldots, d$,
  \begin{equation*}
    {L}^{\TT}({Q}, {P}) {F}_j{(t,q,p)} = 0.
  \end{equation*}
  Hence,
  \begin{align*}
    {L}^{\TT}({Q}, {P}) &\partial_{z_k} \bigl((\wt{\sigma}_{0,\delta}
    {F}_j Z_{jk}^{-1}){(t,q,p)} L^\TT(q,p) \chi_{\delta}(q, p)\bigr)\\
    &= \wt{\sigma}_{0,\delta} {(t,q,p)}{L}^{\TT}({Q}, {P})
    (\partial_{z_k} {F}_j Z_{jk}^{-1}){(t,q,p)} L^\TT({q}, {p})
    \chi_{\delta}(q, p) \\
    & = - \wt{\sigma}_{0,\delta}{(t,q,p)} (
    \partial_{z_k} {L})^{\TT}({Q}, {P}) ({F}_j Z_{jk}^{-1}){(t,q,p)}
    L^\TT({q},{p})\chi_{\delta}(q, p).
  \end{align*}
  Therefore, \eqref{eq:1st} is implied by
  \begin{equation*}
    \begin{aligned}
      \partial_t \wt{\sigma}_{0,\delta}{(t,q,p)}& \chi_{\delta}(q, p)
      + \wt{\sigma}_{0,\delta}{(t,q,p)} {L}^{\TT} (\partial_{{P}} H
      \cdot \partial_{{Q}} {R} - \partial_{{Q}} H \cdot \partial_{{P}}
      {R})\chi_{\delta}(q, p) \\
      & + \wt{\sigma}_{0,\delta}{(t,q,p)} (\partial_{z_k} {L})^{\TT}
      \bigl({F}_j Z_{jk}^{-1}\bigr){(t,q,p)}\chi_{\delta}(q, p) \\
      & + \wt{\sigma}_{0,\delta}{(t,q,p)} \partial_{z_s} Q_j
      Z_{ks}^{-1}{(t,q,p)} \\
      & \hspace{6em} \times{L}^{\TT} \bigl( -
      \partial_{j} A_k(Q) + \frac{\I}{2} P_l \partial^2_{jk}
      A_l(Q) \bigr) \chi_{\delta}(q, p) = 0,
    \end{aligned}
  \end{equation*}
  which is valid by \eqref{eq:wtsigma_0} and \eqref{eq:lambda}.  Here
  in the last equation, $H$, $L$ and $R$ are evaluated at $(Q,P)$.

  Similarly, we can show that $v_{t, k, \delta}{(x)} = 0$ for higher
  orders, by using \eqref{eq:wtMk}-\eqref{eq:wtsigma_k}.

  Therefore, to prove the Lemma, it remains to bound the remainder
  $r_{t, K, \delta}{(x)}$. Using Proposition~\ref{prop:L2bound}, it
  suffices to control the $L^{\infty}$ norm of $(\wt{L}_{n,
    k}\wt{M}_{n, s, \delta}){(t,q,p)}$ for $k = 1, \cdots, 2K + 2$ and $s = 0,
  \cdots, K-1$.  This in turn follows from Lemma~\ref{lem:wtMbound}
  and the definition of $\wt{L}_{n,k}$.
\end{proof}

To relate $\wt{\mc{P}}_{t, K, \delta}^{\veps}$ with $\mc{P}_{t, K,
  \delta}^{\veps}$ defined in \eqref{eq:ansatz}, the following lemma
shows that they are essentially the same when applying on
asymptotically high frequency initial data.
\begin{lemma}\label{lem:ptilde2}
  For any $T > 0$, $K\in \NN$,
  \begin{equation*}
    \sup_{0 \leq t \leq T} \norm{\mc{P}_{t, K, \delta}^{\veps} u_0^{\veps}
      -  \wt{\mc{P}}_{t, K, \delta}^{\veps} u_0^{\veps}}_{L^2(\RR^d; \CC^N)}
    = \Or(\veps^{\infty}).
  \end{equation*}
\end{lemma}

\begin{proof}
  By definition \eqref{eq:ansatz} and \eqref{eq:ansatzdelta},
  \begin{multline*}
    (\mc{P}_{t,K, \delta}^{\veps} u_0^{\veps} - \wt{\mc{P}}_{t, K,
      \delta}^{\veps} u_0^{\veps} )(x) = \frac{1}{(2\pi \veps)^{3d/2}}
    \sum_{n=1}^N \sum_{k=0}^{K-1}
    \int_{\RR^{3d}} e^{\I\Phi_n/\veps} \\
    \times \veps^k\Bigl(M_{n,k}(t, q, p) \chi_{\delta}(q, p) - \wt{M}_{n, k,
      \delta}(t, q, p)\Bigr) u_{0}(y) \ud q \ud p \ud y.
  \end{multline*}
  From the constructions of $M_{n, k}$ and $\wt{M}_{n, k, \delta}$, it
  is easy to see that for $t \in [0, T]$ and $(q, p) \in K_{\delta}$.
  \begin{equation}\label{eq:MwtM}
    M_{n, k}(t, q, p) = \wt{M}_{n, k, \delta}(t, q, p).
  \end{equation}
  As $\chi_{\delta}(q, p) = 1$ for $(q, p) \in K_{\delta}$, we have
  \begin{equation*}
    M_{n, k}(t, q, p) \chi_{\delta}(q, p) = \wt{M}_{n, k, \delta}(t, q, p).
  \end{equation*}
  Using \eqref{eq:MwtM} with $\delta/2$, we have then
  \begin{equation}\label{eq:MwtM2}
    M_{n, k}(t, q, p) = \wt{M}_{n, k, \delta/2}(t, q, p)
  \end{equation}
  for $(q, p) \in K_{\delta/2}$, and hence in particular, for $(q, p)
  \in \supp \chi_{\delta}$. Combining \eqref{eq:MwtM2} with
  Lemma~\ref{lem:wtMbound} gives
  \begin{equation}\label{eq:boundM}
    \sup_{t \in [0, T]} \sup_{(q, p) \in \RR^{2d}}
    \abs{M_{n, k}(t, q, p) \chi_{\delta}(q, p)} \leq C_T.
  \end{equation}

  Lemma~\ref{lem:wtMbound} guarantees $\supp \wt{M}_{n, k, \delta}(t,
  \cdot, \cdot) \subset K_{\delta/2}$. Since $\supp \chi_{\delta}
  \subset K_{\delta/2}$, we also have $\supp (M_{n, k}(t, \cdot,
  \cdot) \chi_{\delta}) \subset K_{\delta/2}$. Therefore, by \eqref{eq:MwtM}, we have
  \begin{equation}\label{eq:diffsupp}
    \supp \Bigl(M_{n,k}(t, \cdot, \cdot) \chi_{\delta} - \wt{M}_{n, k,
      \delta}(t, \cdot, \cdot)\Bigr) \subset K_{\delta/2} \backslash K_{\delta}.
  \end{equation}

  Using a similar argument as in the proof of
  Proposition~\ref{prop:L2bound}, one has
  \begin{equation*}
    \begin{aligned}
      \Bigl\lVert \mc{P}_{t, K, \delta}^{\veps} & u_0^{\veps} -
      \wt{\mc{P}}_{t, K, \delta}^{\veps} u_0^{\veps}
      \Bigr\rVert_{L^2(\RR^d; \CC^N)} \\
      & \leq 2^{-d/2} \sum_{n=1}^N\sum_{k=0}^{K-1} \veps^k \norm{
        \bigl(M_{n,k}(t, \cdot, \cdot) \chi_{\delta} - \wt{M}_{n, k,
          \delta}(t, \cdot, \cdot)\bigr)
        \ms{F}^{\veps} u_0^{\veps}}_{L^2(\RR^{2d}; \CC^N)} \\
      & \leq 2^{-d/2} \sum_{n=1}^N\sum_{k=0}^{K-1} \veps^k \norm{
        \bigl(M_{n,k}(t, \cdot, \cdot) \chi_{\delta} - \wt{M}_{n, k,
          \delta}(t, \cdot,
        \cdot)\bigr)}_{L^{\infty}(\RR^{2d}; \CC^{N\times N})} \\
      & \hspace{10em} \times \norm{ \ms{F}^{\veps}
        u_0^{\veps}}_{L^2(K_{\delta/2} \backslash K_{\delta}; \CC^N)} \\
      & \leq C_{\delta, T, K} \norm{ \ms{F}^{\veps}
        u_0^{\veps}}_{L^2(K_{\delta/2} \backslash K_{\delta}; \CC^N)},
    \end{aligned}
  \end{equation*}
  where we have used \eqref{eq:diffsupp} in the second inequality, and
  \eqref{eq:boundM} and Lemma~\ref{lem:wtMbound} in the last
  inequality. The proof is concluded by noticing that
  \begin{equation*}
    \norm{ \ms{F}^{\veps}
      u_0^{\veps}}_{L^2(K_{\delta/2} \backslash K_{\delta}; \CC^N)}
    \leq \norm{ \ms{F}^{\veps}
      u_0^{\veps}}_{L^2(\RR^{2d} \backslash K_{\delta}; \CC^N)}
    = \Or(\veps^{\infty}),
  \end{equation*}
  by Definition~\ref{def:highfreq}.
\end{proof}

\section{Proof of the main results}\label{sec:proof}

We recall the energy estimate for linear strictly hyperbolic
system. The proof can be found for example in \cite{Ta:81}.
\begin{lemma}\label{lem:apriori} Given strictly hyperbolic
  system
  \begin{equation*}
    \partial_t u + \sum_{l=1}^d A_l(x) \partial_{x_l} u
    = f,
  \end{equation*}
  with initial condition $u(0, x) = u_0(x)$, where $A_l(x)$ are given
  as in \eqref{eq:hypersys}. For any $T>0$, there exists a constant
  $C_T$ such that
  \begin{equation*}
    \sup_{0\leq t\leq T} \norm{u(t,x)}_{L^2(\RR^d; \CC^N)}^2
    \leq C_T \biggl( \norm{u_0(x)}_{L^2(\RR^d; \CC^N)}^2 + \int_0^T
    \norm{f(s,x)}^2_{L^2(\RR^d; \CC^N)} \ud s \biggr).
  \end{equation*}
\end{lemma}



\begin{prop}\label{prop:term2}
  Under the same assumption of Theorem~\ref{thm:main}, for any $K \in
  \NN$, $T > 0$, there exists constants $C_{T, K}$ and $\veps_0>0$,
  such that for any $\veps \in (0, \veps_0]$,
  \begin{equation}
    \sup_{t \in [0, T]}
    \norm{\mc{P}_t u_{0}^{\veps} - \wt{\mc{P}}_{t,K, \delta}^{\veps}
      u_0^{\veps} }_{L^2(\RR^d; \CC^N)}
    \leq  C_{T, K} M \veps^{K-1}.
  \end{equation}
\end{prop}

\begin{proof}
  We denote
  \begin{equation*}
    e^{\veps}(t, x) = (\mc{P}_t u_{0}^{\veps})(x)
    - (\wt{\mc{P}}_{t,K, \delta}^{\veps}
    u_0^{\veps})(x).
  \end{equation*}
  Proposition~\ref{prop:highorder} implies
  \begin{equation*}
    \sup_{t \in [0, T]} \norm{ ( \partial_t + A_l(x) \partial_{x_l})
      e^{\veps}(t, \cdot)}_{L^2(\RR^d; \CC^N)}
    \leq C_{T, K} M \veps^{K-1}.
  \end{equation*}
  Notice that by the construction of filtered frozen Gaussian
  approximation,
  \begin{align*}
    (\wt{\mc{P}}_{0, K, \delta}^{\veps} u_0^{\veps})(x)& =
    \frac{1}{(2\pi \veps)^{3d/2}} \sum_{n=1}^N \int_{\RR^{3d}} e^{-
      \abs{x - q}^2/(2\veps) + \I p \cdot (x - q)/\veps -
      \abs{y-q}^2/(2\veps) - \I p \cdot (y - q)/\veps } \\
    & \hspace{8em} \times 2^{d/2} R_n(q, p) L_n^{\TT}(q, p)
    \chi_{\delta}(q, p) u_0^{\veps}(y) \ud q \ud p \ud y \\
    & = (\mc{F}^{\veps})^{\ast}( \chi_{\delta} \mc{F}^{\veps} u_0^{\veps}),
  \end{align*}
  where we have used that fact that for $(q, p) \in \RR^{2d}$ with
  $\abs{p} >0$,
  \begin{equation*}
    \sum_{n=1}^N R_n(q, p) L_n^{\TT}(q, p) = \Id_N.
  \end{equation*}
  This implies
  \begin{equation*}
    e^{\veps}(0, x) = u_0^{\veps}(x) -
    (\wt{\mc{P}}_{0, K, \delta}^{\veps} u_0^{\veps})(x)
    = (\mc{F}^{\veps})^{\ast}( (1 - \chi_{\delta}) \mc{F}^{\veps} u_0^{\veps}).
  \end{equation*}
  Hence, using Proposition~\ref{prop:isometry},
  \begin{equation*}
    \norm{e^{\veps}(0, \cdot)}_{L^2(\RR^d; \CC^N)}
    = \norm{(1 - \chi_{\delta}) \mc{F}^{\veps} u_0^{\veps}}_{L^2(\RR^{2d}; \CC^N)}
    \leq \norm{\mc{F}^{\veps} u_0^{\veps}}_{L^2(\RR^{2d}\backslash K_{\delta}; \CC^N)}
    = \Or(\veps^{\infty}).
  \end{equation*}
  The conclusion of the Proposition follows easily from
  Lemma~\ref{lem:apriori}.
\end{proof}

Finally, we conclude with the proof of Theorem~\ref{thm:main}.
\begin{proof}[Proof of Theorem~\ref{thm:main}]
  Triangle inequality gives
  \begin{equation*}
    \begin{aligned}
      \norm{\mc{P}_t u_0^{\veps} - \mc{P}_{t,K, \delta}^{\veps}
        u_0}_{L^2(\RR^d; \CC^N)} \leq & \norm{\mc{P}_t u_0^{\veps} -
        \wt{\mc{P}}_{t,K+1, \delta}^{\veps} u_0^{\veps}}_{L^2(\RR^d; \CC^N)} \\
      & + \norm{ \wt{\mc{P}}_{t,K+1, \delta}^{\veps} u_0^{\veps} -
        \mc{P}_{t, K+1, \delta}^{\veps} u_0^{\veps}}_{L^2(\RR^d;
        \CC^N)} \\
      & + \norm{ \mc{P}_{t,K+1, \delta}^{\veps} u_0^{\veps} -
        \mc{P}_{t, K, \delta}^{\veps} u_0^{\veps}}_{L^2(\RR^d;
        \CC^N)}.
    \end{aligned}
  \end{equation*}
  The first two terms are estimated by Lemma~\ref{lem:ptilde2} and
  Proposition~\ref{prop:term2}. For the last term, notice that by
  definition
  \begin{equation*}
    \mc{P}_{t,K+1, \delta}^{\veps} u_0^{\veps} - \mc{P}_{t, K,
      \delta}^{\veps} u_0^{\veps}
    = \mc{I}_n^{\veps}(t, \veps^K
    M_{n,K}(t, \cdot, \cdot) \chi_{\delta} ) u_0^{\veps},
  \end{equation*}
  and hence, using \eqref{eq:boundM} and
  Proposition~\ref{prop:L2bound}, we have
  \begin{equation*}
    \norm{ \mc{P}_{t,K+1, \delta}^{\veps} u_0^{\veps} -
      \mc{P}_{t, K, \delta}^{\veps} u_0^{\veps}}_{L^2(\RR^d;
      \CC^N)} \leq  C_{K, T} M \veps^K.
  \end{equation*}
  The Theorem is proved.
\end{proof}

\bibliographystyle{amsalpha}
\bibliography{fga}

\end{document}